\documentclass[11pt]{amsart}

\usepackage{mystyle}

\title{Scattering on singular Yamabe spaces}
\author[Sun-Yung Alice Chang, Stephen E. McKeown, and Paul Yang]{Sun-Yung Alice Chang\textsuperscript{*}, Stephen E. McKeown\textsuperscript{**}, and Paul Yang\textsuperscript{*,\dag}}
\dedicatory{In honor of Antonio C\'{o}rdoba and Jos\'{e} Luis Fern\'{a}ndez}
\address{Department of Mathematics, Princeton University, Princeton, NJ 08544, USA}
\address{Department of Mathematical Sciences, University of Texas at Dallas, 800 W. Campbell Road, Richardson, TX 75080, USA}
\address{Department of Mathematics, Princeton University, Princeton, NJ 08544, USA}
\email{chang@math.princeton.edu}
\email{stephen.mckeown@utdallas.edu}
\email{yang@math.princeton.edu}
\date{\today}

\thanks{*Partially supported by NSF grant DMS-1607091. **Partially supported by NSF RTG grant DMS-1502525. \dag Partially supported by Simons Foundation grant 615589.}

\begin{document}
\maketitle
\begin{abstract}
	We apply scattering theory on asymptotically hyperbolic manifolds to singular Yamabe metrics, applying the results to the study of the conformal geometry of compact manifolds with boundary.
	In particular, we define extrinsic versions of the conformally invariant powers of the Laplacian, or GJMS operators, on the boundary of any such manifold, along with associated extrinsic $Q$-curvatures.
	We use the existence and uniqueness of a singular Yamabe metric conformal to define also nonlocal extrinsic fractional GJMS operators on the boundary, and draw other global conclusions about the scattering operator,
	including a Gauss-Bonnet theorem in dimension four.
\end{abstract}
\renewcommand*{\thetheorem}{\Alph{theorem}}
\section{Introduction}
Scattering theory on asymptotically hyperbolic manifolds has been studied with great profit in the case that the metric is Einstein. In this paper we study the scattering problem in the case
that the metric has constant scalar curvature. For every compact Riemannian manifold $(X^{n + 1},\bar{g})$ with boundary $M$, it is known that there is precisely one defining function $u$ for the boundary so that
the \emph{singular Yamabe metric} $g = u^{-2}\bar{g}$ has constant scalar curvature, so this study can be seen as bringing scattering theory to bear as a tool
in the study of the conformal geometry of compact manifolds with boundary.
We apply the methods of \cite{gz03,fg02,cqy08} and others to the setting of \cite{gw15,g17}, which used the singular Yamabe problem to study conformal hypersurfaces.

Given an asymptotically hyperbolic (AH) manifold $(X^{n + 1},g)$ with boundary $M^n = \partial X$, the scattering problem is defined as follows. Let $\bar{g} = r^2g$ be a geodesic compactification
of $g$, i.e., suppose that $|dr|_{\bar{g}} = 1$ on a neighborhood of $M$. Let $f \in C^{\infty}(M)$. Consider the equation
\begin{equation}
	\label{scateqintro}
	(\Delta_g + s(n - s))v = 0,
\end{equation}
where we assume that $s > \frac{n}{2}$ and that $s(s - n) \notin \sigma_{pp}(\Delta_g)$ (the set of $L^2$ eigenvalues of the Laplacian); in our convention,
$\Delta_g$ is a negative operator. It is shown in \cite{mm87} that a solution $v$ to this equation has asymptotics
\begin{equation*}
	v = r^{n - s}F + r^sG,
\end{equation*}
where $F, G \in C^{\infty}(X)$, at least so long as $2s - n \notin \mathbb{Z}$. In \cite{gz03}, it is shown that we can always solve (\ref{scateqintro}) with $F|_M = f$. The scattering operator
$S(s):C^{\infty}(M) \to C^{\infty}(M)$ is then defined by $S(s)f = G|_M$, and this operator extends to be meromorphic on $\re s \geq \frac{n}{2}$, with poles only at $s$ such that
$s(n - s)$ is an eigenvalue of $\Delta_g$ or at $s = \frac{n + q}{2}$, $q \in \mathbb{N}$.

The case of $g$ Einstein has been thoroughly studied. In that case, the poles for $s = \frac{n + q}{2}$ with $q$ odd actually do not exist -- $S(s)$ is regular at these points -- while at
$s = \frac{n}{2} + j$, the operator $S(s)$ has a simple pole whose residue is the so-called GJMS operator $P_{2j}$ (with $j \leq n/2$ if the boundary dimension $n$ is even). If $n$ is even, $S(s)1$ is holomorphic across 
$s = n$, and in fact $S(n)1 = c_{n}Q_n$, where $Q_n$ is the $n$th-order $Q$-curvature and $c_n$ is a universal constant. On the other hand, if $n$ is odd, then $S(n)1 = 0$, which reflects
the fact that there is no (locally-defined) $Q$-curvature of odd order. However, a nonlocal $Q$-like term was defined in \cite{fg02} by taking $Q_n = k_n\frac{d}{ds}|_{s = n}S(s)1$ in case $n$ is odd.
In addition to these results, there have been numerous interesting results linking the scattering operator with the renormalized volume of an AH Einstein manifold. 

We recall the definition of the renormalized volume in that setting. Let $r$ be a geodesic defining function for $M$, and consider the expansion
\begin{equation}
	\label{volexp}
	\int_{r > \varepsilon}dv_g = c_0\varepsilon^{-n} + c_1\varepsilon^{1 - n} + \cdots + c_{n - 1}\varepsilon^{-1} + \mathcal{E}\log\left( \frac{1}{\varepsilon} \right) + V_+ + o(1).
\end{equation}
The quantity $V_+$ is called the \emph{renormalized volume}. While \emph{a priori} it depends on the choice of geodesic defining function $r$, it is shown in \cite{hs98,g99} that, if $n$ is odd, then $\mathcal{E} = 0$ and
$V_+$ is in fact well-defined independent of $r$. If the boundary dimension $n$ is even, on the other hand, then $\mathcal{E}$ is well-defined independent of $r$. Using scattering theory, it was shown in \cite{gz03,fg02} that
in the even case, $\mathcal{E} = \int_MQ_ndv_k$ (where $k = \bar{g}|_{TM}$). In \cite{cqy08}, a formula was given relating the renormalized volume in the even-$n$ case to an integral over $M$ of a non-local quantity defined in terms
of the scattering operator.

It was shown already in \cite{g17} that the renormalized volume can be defined as well in the singular Yamabe case, where $r$ now is the $\bar{g}$-distance to the boundary and $g$ is the corresponding singular Yamabe
metric; but in this case, $\mathcal{E}$ is generically nonvanishing in all dimensions and is a conformal invariant, while $V_+$ is no longer invariant. In fact, in the case $n = 2$, $\mathcal{E}$ is simply a linear
combination of the Willmore energy and the Euler characteristic. Thus, in case $n > 2$, $\mathcal{E}$ can be seen as a generalization of the Willmore energy.

We generalize the scattering picture to the singular Yamabe setting, and obtain extrinsic analogues to all of the above results. 
Unlike the Einstein setting, where the existence and uniqueness of Einstein metrics remain extremely challenging and one is often
constrained in practice to use a non-unique formal expansion, the analytic picture in the singular Yamabe case is well understood and perhaps optimal: given any compact Riemannian manifold $(X^{n + 1},\bar{g})$ with boundary, there
is a unique defining function $u$ for the boundary so that the AH metric $g = u^{-2}\bar{g}$ has constant scalar curvature $-n(n + 1)$. Any global quantities defined in terms of this 
singular Yamabe metric are thus well defined for $(X,\bar{g})$.
This comes with a price, however: the power series expansion of $u$ at the boundary $M$ is not smooth beyond order $n + 1$, after which $\log$ terms appear. 
Throughout, therefore, we must keep careful track of the expansion of $u$ and its regularity
in a way that has not had to be done in the Einstein case, where smooth formally Einstein metrics have largely been considered.

Also unlike AH Einstein spaces, singular Yamabe metrics have no particular parity properties,
and this has numerous consequences: in general, this theory tends to behave more like the even-dimensional-boundary Einstein theory in all dimensions than like the odd-dimensional Einstein theory, whose special behaviors are entirely a product
of parity considerations. In some cases, though, even features of the even-dimensional Einstein case are not reproduced.
A first consequence, for example, is that $S(s)$ generically has poles at $s = \frac{n + q}{2}$ for all $q$ (not only even $q$). Therefore, we get residue operators $P_q^{SY} = \res_{s = n/2 + q/2}S(s)$ for all $q \leq n$.
Of course, since Einstein metrics are the singular Yamabe metrics in their conformal class, our results must in all case reduce to the usual ones in that special case.
Similarly, for all $n$ -- not only even -- we can define $\syQ_n = c_n^{-1}S(n)1$. Here, $\syQ_n$ is a locally determined curvature quantity along $M$, which however depends on $\bar{g}$ and not merely $\bar{g}|_{TM}$. Similarly, for each $q$, $P_q^{SY}$ is a conformally covariant
differential operator along $M$ whose coefficients depend on extrinsic data.

We now state our main results.
\begin{theorem}
	\label{PSYthm}
	Let $(X^{n + 1},\bar{g})$ be a compact Riemannian manifold with boundary $(M^n,k)$, and let $g$ be the singular Yamabe metric associated to $\bar{g}$. Let $\mathring{L}$ be the tracefree
	second fundamental form of $M$ with respect to $X$, computed with inward-pointing unit normal vector. Then the scattering operator 
	$S(s):C^{\infty}(M) \to C^{\infty}(M)$ extends to a meromorphic family on the strip $-\frac{1}{2} < s < n + \frac{1}{2}$, regular for $\re s = \frac{n}{2}$. It has poles only
	at $s = \frac{n + q}{2}$ (for $q \in \mathbb{N}$) and for such $s$ that $s(s - n) \in \sigma_{pp}(\Delta_g)$, the set of $L^2$ eigenvalues of $\Delta_g$. 
	
	Suppose that $q \in \mathbb{N}$ satisfies
	$2 \leq q \leq n$ and that $\frac{q^2 - n^2}{4} \notin \sigma_{pp}(\Delta_g)$. Then there is a conformally covariant differential operator $P_q^{SY}:C^{\infty}(M) \to C^{\infty}(M)$, satisfying
	\begin{equation*}
		c_qP_q^{SY} = -\res_{s = \frac{n + q}{2}}S(s),
	\end{equation*}
	where $c_q$ is a nonvanishing universal constant, and $P_q$ has principal part $(-\Delta_k)^q$ if $q$ is even and, if $q > 1$ is odd, has the same principal part as
	$\mathring{L}^{\mu\nu}\nabla_{\mu}^k\nabla^k_{\nu}(\Delta_k)^{\frac{q - 3}{2}}$. The operator $P_q^{SY}$ depends only on the jet of $\bar{g}$ at $M$.

	If $\tilde{\bar{g}} = e^{2\omega}\bar{g}$ is a conformally-related metric with corresponding operators $\widetilde{P}_q^{SY}$, then
	\begin{equation*}
		\widetilde{P}_q^{SY}f = e^{-\frac{(n + q)\omega}{2}}P_q^{SY}\left( e^{\frac{(n - q)\omega}{2}}f \right).
	\end{equation*}
\end{theorem}

Several remarks are in order regarding this statement.

First, note that in the Einstein case considered by \cite{gz03}, the integer $q$ is assumed to be even, since otherwise $P_q^{SY}$ vanishes identically in that setting. Moreover, the indexing of our
	operators $P_q^{SY}$ differs from the indexing in that paper by a factor of two -- what we call $P_2$ was there $P_1$, etc. This is because in the other case,
	what are here the odd-order operators vanish. Already in the literature, there is some difference
	in the numbering of these operators. The operators defined in \cite{gz03}, of course (unlike those here) are the GJMS operators. The operators defined here differ from the GJMS operators
	by terms depending on the extrinsic geometry of $M$ in $X$.

	Next, since in the Einstein case, the odd-order operators vanish, the notation $P_q$ for $q$ odd has there developed another meaning that, unfortunately, is not analogous to that used here: see 
	\cite{fg02}. Given the strong analogy between the even-order operators and those in the Einstein case, and the equivalence of definition here of the even- and odd-order operators, it seems there is no
	very good way around having confusing notation. Observe also that the odd-order operators defined in those settings do not exist in this setting, since their definition depends on the vanishing
	of the residue at $s = \frac{n + q}{2}$. A similar remark applies to the definition, below, of $Q_n^{SY}$ for odd $n$; again, see \cite{fg02}.

	As discussed more fully in section \ref{globsec}, the restriction on $s$ in the statement of the meromorphicity of the scattering operator is due to the non-regularity at order $n+2$ of
	the singular Yamabe solution $u$. In fact, stronger statements could be made, but the nonregularity makes it somewhat delicate to discuss and define just what this means. Since
	we are not interested behavior at higher $s$ in any event, we give the weaker statement.

	Finally, in work such as \cite{cg11}, it is common to use the notation $s = \frac{n}{2} + \gamma$. The relationship between $q$ and $\gamma$, then, is $q = 2\gamma$.

	\begin{corollary}
		The operators $P_q^{SY}$ are self-adjoint.
	\end{corollary}
	This follows from their definition in terms of the scattering operator.

	We now define the $Q_n^{SY}$-curvature, as in \cite{gz03}, by $Q_n^{SY} = c_{n}^{-1}S(n)1$. This is well-defined since $\mathbb{R} \subseteq \ker P_{n}^{SY}$. This \emph{extrinsic} $Q$-curvature
	quantity follows a conformal transformation law like its intrinsic cousin (which however exists only for even order) and, like it, gives the conformally invariant term in the volume expansion:

	\begin{theorem}
		\label{thmQ}
		If $\mathcal{E}$ is as in (\ref{volexp}), interpreted as in \cite{g17}, then
		\begin{equation*}
			\mathcal{E} = \oint_MQ_n^{SY}dv_k.
		\end{equation*}

		Furthermore, if $\tilde{\bar{g}} = e^{2\omega}\bar{g}$, then
		\begin{equation*}
			e^{n\omega}\widetilde{Q}_n^{SY} = Q_n^{SY} + P_n^{SY}\omega.
		\end{equation*}
	\end{theorem}

	A corollary then follows immediately from Theorem 3.1 of \cite{g17}. To state it, we must say slightly more about the singular Yamabe function. Given $\bar{g}$, there is a unique defining function $u$ for $M$
	such that $g = u^{-2}\bar{g}$ has scalar curvature $-n(n + 1)$. As discussed above and in \cite{g17}, $u$ has the expansion (in terms of the distance function $r$ from the boundary $M$ in $X$)
	\begin{equation}
	\label{usyexp}
	u = r + u_2r^2 + u_3r^3 + \cdots + u_{n + 1}r^{n + 1} + \mathcal{L}r^{n + 2}\log(r) + u_{n + 2}r^{n + 2} + o(r^{n + 2}),
	\end{equation}
	where $\mathcal{L}$ is a locally and extrinsically defined function on $M$, conformally invariant of weight $-(n + 1)$. With this notation fixed, we have the following:
	
\begin{corollary}
	\label{corQ}
	Let $(X,\bar{g})$ be a Riemmanian manifold, and let $F_t:M \hookrightarrow X, 0 \leq t < \delta$, be a smoothly varying variation of $M$, where $F_0$ is the identity. Let $M_t = F_t(M)$, and
	$\overline{n}$ the inward-pointing unit normal to $M$ in $X$. Then
	\begin{equation*}
		\left.\frac{d}{dt}\right|_{t = 0}\oint_{M_t}\syQ_tdv_{M_t} = (n + 2)(n - 1)\oint_M\langle\dot{F},\overline{n}\rangle_{\bar{g}}\mathcal{L}dv_k,
	\end{equation*}
	where $\mathcal{L}$ is as in (\ref{usyexp}).
\end{corollary}

	The results so far stated about $P_q^{SY}$ and $Q_n^{SY}$ have actually appeared previously in the recent literature. They were derived in a somewhat different framework by Gover and Waldron
	(\cite{gw14,gw15,gw17}), where they are defined in terms of the tractor calculus and so-called \emph{Laplace-Robin operators}. When the current paper was quite advanced in preparation,
	the paper \cite{jo21} of Juhl and Orsted was brought to our attention. Among other things, it reinterprets the results of Gover and Waldron in the setting of scattering theory shared by this paper.
	Our contribution with respect to this part of our material is thus a different treatment, more similar in spirit to \cite{gz03,fg02,cqy08}.
	On the other hand, our perspective here is also focused more than those papers on the unique global scattering operator associated to every
	compact Riemmanian manifold with boundary. In particular, rather than doing only asymptotic analysis at the boundary and thus neglecting the impact of the non-regularity of the singular Yamabe solution,
	we take account of the logarithmic term. The following results, more global in spirit, are new to this paper.

	First, as stated above, $S(s)1$ is smooth across $s = n$. We now define 
	\begin{equation}
		\label{Sdef}
		\mathscr{S} = \left.\frac{d}{ds}\right|_{s = n}S(s)1 \in C^{\infty}(M).
	\end{equation}
	This is a function dependent on the global geometry of $(X,\bar{g})$.
	Next, it follows from \cite{gz03} that for $s \in \mathbb{C}$ near 1, there exists $u_s = r^{n - s}F_s + r^sG_s$ satisfying $(\Delta_g + s(n - s))u_s = 0$ and
	$F_s|_{M} = 1$. Moreover, $F_s$ and $G_s$ may be uniquely determined by the requirement that they be holomorphic in $s$. We then write
	\begin{equation*}
		F_s(r) = 1 + a_1(s) r + a_2(s)r^2 + \cdots + a_n(s)r^n + \cdots,
	\end{equation*}
	where each $a_j$ is a smooth function on $M$, and $a_j$ is locally (and extrinsically) determined for $1 \leq j \leq n$. Finally, we recall the definition in \cite{g17} (see also \cite{g99}) of
	the \emph{renormalized volume coefficients}. Since $g = u^{-2}\bar{g}$ for a defining function $u$ for the boundary, we may write
	\begin{equation*}
		dv_g = r^{-1-n}(1 + v^{(1)}r + v^{(2)}r^2 + \cdots)drdv_k
	\end{equation*}
	on a collar neighborhood $[0,\varepsilon)_r \times M$ near $M$ in $X$. The functions $v^{(j)} \in C^{\infty}(M)$ are the renormalized volume coefficients. With these notations in hand, we may state
	our next result, which is analogous to a theorem proved in \cite{cqy08} (based on \cite{fg02}) in the Einstein setting.
\begin{theorem}
	\label{globthm1}
	Suppose $(X^{n + 1},\bar{g})$ is a Riemannian manifold with boundary, and that $g = u^{-2}\bar{g}$ is the associated singular Yamabe metric. Let 
	$V(X,g,\bar{g})$ be the renormalized volume of $g$ when computed with respect to $\bar{g}$, i.e. the constant term in the expansion
	(\ref{renormvol}). Then
	\begin{equation}
		\label{volscateq}
		\begin{split}
			V(X,g,\bar{g}) =& -\oint \mathscr{S}dv_k\\
			&\quad -\frac{1}{n}\left(\oint_Ma_1'v^{(n - 1)}dv_k +  2\oint_Ma_2'v^{(n - 2)}dv_k + \cdots +\right. \\
			&\quad\left.(n - 1)\oint_Ma'_{n - 1}v^{(1)}dv_k + n\oint_Ma_n'dv_k \right),
		\end{split}
	\end{equation}
	where a prime denotes $\frac{d}{ds}|_{s = n}$.
\end{theorem}
	All the terms on the right-hand side except the first are local.

	It was shown in \cite{fg02} that, if $g$ is in fact Einstein and $n$ is odd, then this formula holds with only the first term on the right-hand side present.
	Although, in such a case, $g$ is also the singular Yamabe metric and thus (\ref{volscateq}) applies, the results are not inconsistent, because in that case,
	the $v^{(k)}$ with $k$ odd and the $a_k$ with $k$ odd are both identically zero, and one of these is a factor in each of the integrands multiplied above by $\frac{1}{n}$ when $n$ is odd.

	The following theorem, which was proved in \cite{cqy08} in the the special case of even-dimensional Einstein metrics, states (in all dimensions) that the scattering term in (\ref{volscateq}) is a conformal primitive
	for total $Q^{SY}$.

\begin{theorem}
	\label{globthm2}
	Let $\bar{g}$ be a smooth metric on $(X^{n + 1},M^n)$, and $g$ the corresponding singular Yamabe metric. As usual, let $k = \bar{g}|_{TM}$. Suppose $\omega \in C^{\infty}(X)$. Then
	\begin{equation*}
		\left.\frac{d}{d\alpha}\right|_{\alpha = 0}\oint_M\mathscr{S}_{e^{2\alpha\omega}\bar{g}}dv_{e^{2\alpha\omega}k} = -2c_{n}\oint_M \syQ_n\omega dv_k.
	\end{equation*}
\end{theorem}

Finally, for $X$ of dimension four, we apply Theorem \ref{globthm1} and the main result of \cite{gg19} to obtain a Gauss-Bonnet theorem in terms of the scattering operator.

\begin{theorem}
	\label{gbprop}
	Let $(X^4,\bar{g})$ be a compact Riemannian manifold with boundary $M^3 = \partial X$, and $g$ the singular Yamabe metric. 
	Let $r$ be the $\bar{g}$-distance function to the boundary, $E$ the Einstein tensor of $X$, $W$ the Weyl tensor, and $\mathscr{S}$ as in (\ref{Sdef}). Then
	\begin{align*}
		8\pi^2\chi(X) &= \frac{1}{4}\int_X|W|_g^2dv_g - \frac{1}{2}f.p.\int_{r > \varepsilon} |E|_g^2 dv_g + \oint_M\left( -6\mathscr{S} + \mathcal{C} \right)dv_k,
	\end{align*}
	where
	\begin{align*}
		\mathcal{C} &= -\frac{11}{36}HR + \frac{1}{108}H\overline{R} + \frac{5}{108}H^3 + \frac{389}{144}H|\mathring{L}|_k^2 + \frac{1}{4}\nabla^{\mu}\nabla^{\nu}\mathring{L}_{\mu\nu}\\
		&\quad + \frac{23}{6}\mathring{L}^{\mu\nu}\overline{R}_{\mu\nu}-\frac{17}{3}\mathring{L}^{\mu\nu}R_{\mu\nu} + \frac{1}{12}\partial_r\overline{R} - \frac{2}{3}\mathring{L}^3.
	\end{align*}
	Here $f.p.$ denotes the finite part of the integral as $\varepsilon \to 0$, $L$ the second fundamental form of $M$, $H = k^{\mu\nu}L_{\mu\nu}$ its mean curvature, and $R, \overline{R}, R_{\mu\nu},
	\overline{R}_{\mu\nu}$ the scalar and Ricci curvatures of, respectively, $M$ and $X$.
\end{theorem}

The following corollary also follows from Corollary 1.4 of that paper, or from the previous result.
\begin{corollary}
	\label{gbcor}
	Let $(X^4,\bar{g})$ be a compact Riemannian manifold with \emph{umbilic} boundary $M^3 = \partial X$. Then the quantity
	\begin{equation*}
		\widetilde{V} = \oint_M\left( -\mathscr{S} + \frac{1}{6}\mathcal{C} \right)dv_k
	\end{equation*}
	is a conformal invariant.
\end{corollary}

We point out also that having a uniquely defined scattering operator, as stated in Theorem \ref{PSYthm}, allows one to define unique \emph{fractional} extrinsic GJMS operators $P_{2\gamma}^{SY} = S(\frac{n}{2} + \gamma)$ 
of order $2\gamma$ as well, as in
\cite{cg11}. In the intrinsic case, such operators are unique only when a global Einstein metric can be found. As there, these operators are nonlocal.

The paper is organized as follows.
In section \ref{backsec} we review the background necessary for the paper, and also introduce geodesic coordinates for the singular Yamabe setting. These are a useful tool for studying
the singular Yamabe metric, and in particular performing computations. In section \ref{locsec}, we develop the existence theory for the class of singular Yamabe GJMS operators of integral order, which entails
formal analysis of the scattering operator at the boundary. The analysis is a variation of that carried out in section four of \cite{gz03}. The result is summarized in Theorem \ref{expandprop}. In section
\ref{globsec}, we turn to the global existence of the scattering operator and the results that follow from it, including Theorems \ref{PSYthm} - \ref{globthm2} and their corollaries. 
In section \ref{compsec}, we perform specific computations in low dimensions. 
This is done both to illuminate and illustrate the theory, and to demonstrate the usefulness of geodesic normal coordinates for carrying out computations with singular Yamabe metrics. We also there prove
Theorem \ref{gbprop}.

Appendix A contains more thorough discussion of some analytic ramifications of the limited regularity of the singular Yamabe metric.

\textbf{Acknowledgements} The authors found an error in their formula for $c_k$ in section \ref{locsec} of an earlier draft thanks to an explicit computation in \cite{jo21}, and have corrected it here.
Stephen McKeown carried out part of the work while a postdoctoral research associate at Princeton University, whom he thanks for the hospitality and support.
\renewcommand*{\thetheorem}{\thesection.\arabic{theorem}}

\section{Background}
\label{backsec}

\subsection{The Singular Yamabe Metric}

Let $\left( X^{n + 1},\bar{g} \right)$ be a Riemannian manifold with boundary $M^n = \partial X$. The singular Yamabe (or Loewner-Nirenberg) problem is to find a defining function $u$ for $M$ so that the conformally related
complete metric $g = u^{-2}\bar{g}$ has constant scalar curvature $R(g) = -n(n + 1)$. It has long been known that the solution $u$ exists and is unique (see \cite{ln74,am88,acf92}). 
Moreover, it was shown in \cite{acf92} that $u$ is regular
in the sense that, if $r(x)$ is the distance function to $M$ on $X$ with respect to $\bar{g}$, 
then $u$ has an asymptotic expansion in powers of $r$ and $r^k\log(r)^j$, where $k \geq n + 2$. As discussed in detail in \cite{g17}, this formal expansion
can be obtained term by term by writing out the equation $R(g) = -n(n + 1)$ in terms of $\bar{g}$. The equation becomes
\begin{equation}
	\label{usyeq}
	n(n + 1) = n(n + 1)|du|_{\bar{g}}^2 - 2nu\Delta_{\bar{g}}u - u^2\overline{R},
\end{equation}
where $\overline{R} = R(\bar{g})$ is the scalar curvature associated to the metric $\bar{g}$. 
Then, differentiating equation (\ref{usyeq}) term by term, one can write the expansion (\ref{usyexp}),
where $u_2,\cdots,u_{n + 1}$ and $\mathcal{L}$ are locally determined smooth functions on $M$, while $u_{n + 2}$ is globally determined. Each of the locally determined quantities is a universal
expression in the intrinsic and extrinsic geometry of $M$ as a hypersurface of $(X,\bar{g})$. In particular, $\mathcal{L}$ is an extrinsic conformal invariant of weight $-(n + 1)$. The function $u$ itself is also
conformally invariant of weight $1$: if $\breve{\bar{g}} = e^{2\omega}\bar{g}$, then $\breve{u} = e^{\omega}u$. This follows easily from the uniqueness of the singular Yamabe metric.

In \cite{g17} (see also \cite{gw17}), 
it was observed that one can define a renormalized volume for singular Yamabe metrics, and that the volume expansion defines a geometrically interesting energy term $\mathcal{E}$ that in some respects
generalizes the Willmore energy. Specifically, with the above notation, consider the quantity $\vol_g(\left\{ r > \varepsilon \right\})$. This can be expanded in powers of $\varepsilon$ as follows:
\begin{equation}
	\label{renormvol}
	\vol_{g}\left( \left\{ r > \varepsilon \right\} \right) = c_0\varepsilon^{-n} + c_1\varepsilon^{-n + 1} + \cdots + c_{n - 1}\varepsilon^{-1} + \mathcal{E}\log\frac{1}{\varepsilon} + V + o(1).
\end{equation}
Here each $c_k$ is determined locally in the sense that it is an integral over $M$ of locally, extrinsically-determined quantities. The energy $\mathcal{E}$ is a global (extrinsic) conformal invariant of the
embedding of $M$ in $X$, and is also the integral over $M$ of a local term; 
and $V$ is globally determined in the sense that it may depend on the geometry of $(X,\bar{g})$ far away from $M$. When $n = 2$, it was shown in the same paper that $\mathcal{E}$ is a linear combination
of the Willmore energy of $M$ and its Euler characteristic.

The singular Yamabe metric is, among other things, an asymptotically hyperbolic (AH) metric, but there are subtleties that make the application of AH theory to this situation slightly subtle. It will be useful to
discuss those subtleties here. We recall the following definition.
\begin{definition}
	An asymptotically hyperbolic space is a compact manifold $X^{n + 1}$ with boundary $M^n$, equipped on the interior $\mathring{X}$ with a metric $g$ such that, for any defining function $\varphi$ for $M$ in $X$,
	$\varphi^2g$ extends to a metric $\bar{g}$ on $X$ and $|d\varphi|_{\bar{g}}^2|_{TM} \equiv 1$. If $\varphi$ is smooth, then $g$ is called $C^k$ (or smooth) AH if $\bar{g}$ is a $C^k$ (or smooth) compact metric.
	The conformal infinity is the conformal class $[\varphi^2g|_{TM}]$ on $M$.
\end{definition}
In the most typical AH settings, one is given or constructs an AH metric, and it is this that is considered natural; various compactifications correspond to various defining functions $\varphi$, but there is no canonical defining
function and thus no canonical compactification. On the other hand, in the singular Yamabe problem, the problem data is precisely the compactification $\bar{g}$, which is taken to be smooth. The AH metric, which is the singular
Yamabe metric, is canonically obtained from $\bar{g}$, but is not generically a smooth AH metric since $u$ is not generically a smooth function; in particular, $u$ is typically only $C^{n + 1}$ and
polyhomogeneous, and it is easy to show that $g$ is only a $C^{n}$ AH metric.

A very useful result in AH geometry is the normal-form theorem, first proved in \cite{gl91}. We will require it here with more attention than usual to the regularity of the metric, so we here give a statement
suited to our needs.

\begin{lemma}
	\label{geodlem}
	Let $\bar{g}$ be a smooth metric on the manifold $X^{n + 1}$ with boundary $M$, and $k = \bar{g}|_{TM}$. 
	Let $g = u^{-2}\bar{g}$ be the corresponding singular Yamabe metric. Then for $\varepsilon > 0$ sufficiently small there is a unique $C^n$ diffeomorphism
	$\psi:[0,\varepsilon)_{\hat{r}} \times M \hookrightarrow X$ onto a collar neighborhood of $M$ such that $\psi^*g = \frac{d\hat{r}^2 + h_{\hat{r}}}{\hat{r}^2}$ with $h_0 = k$.
	
	Moreover, $(\psi^{-1})^*\hat{r} \in C^{\infty}(\mathring{X}) \cap C^{n + 1}(X)$; we will hereafter denote this function simply $\hr$. It further satisfies $r\hr \in C^{n + 2}$ and $\lim_{r \to 0}\partial_r^{n + 1}\hr = 0$.
\end{lemma}
The proof is in appendix \ref{proofapp}.

If $\hat{r}$ is extended to $X$ as a positive smooth function, we will call the metric $\hat{\bar{g}} = \hat{r}^2g$ a \emph{geodesic representative associated to $\bar{g}$}. The importance of these metrics for us is that
the singular Yamabe function associated to one of them is simply $\hat{r}$ itself -- that is, for such a metric, the intrinsic distance to the boundary is the solution to the singular Yamabe equation (\ref{usyeq}). This will
greatly simplify some of our computations in section \ref{compsec}.

\subsection{Scattering on Asymptotically Hyperbolic Spaces}

The main results we need come from the paper \cite{gz03} by Graham and Zworski, which analyzed scattering on asymptotically hyperbolic manifolds using tools from \cite{mm87}. Let $(X^{n + 1},g)$ be an AH manifold, with $M = \partial X$
and conformal infinity $[h]$.
Let $x$ be a defining function for the boundary and $\bar{g} = x^2g$ (we do not here assume that $x$ is a geodesic defining function).
Consider the operator $(\Delta_g + s(n - s))u = 0$, where $u \in C^{\infty}(\mathring{X})$. It is easy to show, by writing the operator $\Delta_g$ in local coordinates, that
any solution to the equation must have leading order $x^{n - s}$, assuming $\re s \geq \frac{n}{2}$ and $s \neq \frac{n}{2}$. 
Thus, the problem considered in \cite{gz03} is the following, for $\re s \geq \frac{n}{2}$ with $s \neq \frac{n}{2}$.
Let $f \in C^{\infty}(M)$ be prescribed. Then consider the problem
\begin{equation}
	\label{probstat}
	\begin{split}
	     (\Delta_g + s(n - s))u &= 0\\
	     u &= x^{n - s}F + o(x^{n - s})\text{ if } \re s \neq \frac{n}{2}\\
	     u &= x^{n - s}F + x^sG + O(x^{n/2 + 1})\text{ if } \re s = \frac{n}{2}, s \neq \frac{n}{2}\\
	     F, G \in C^{\infty}(X)\text{ with } F|_M = f.
        \end{split}
\end{equation}

To describe the results of the paper, it will be useful to use sections of the normal density bundles $C^{\infty}(M,|N^*M|^s)$ over the boundary, which helpfully parametrize first-order changes in the defining function.
Given a choice $k \in [k]$ of conformal representative of the conformal infinity, and letting $x$ be any defining function such that $x^2g|_{TM} = k$, we can trivialize $|N^*M|^s$ by the global section
$|dx|^s$, and in particular can identify $C^{\infty}(M,|N^*M|^s)$ with $C^{\infty}(M)$. We will also use the notation $\mathcal{E}(-s)$ for the bundle $C^{\infty}(M,|N^*M|^s)$.

Let $\sigma(\Delta_g)$ be the spectrum of the Laplacian of $g$. Graham and Zworski proved the following theorem.
\begin{theorem}
	There is a unique family of Poisson operators
	\begin{equation*}
		\mathcal{P}(s):C^{\infty}(M,|N^*M|^{n - s}) \to C^{\infty}(\mathring{X})
	\end{equation*}
	for $\re s \geq n/2, s \neq n/2$, which is meromorphic in $\left\{ \re s > \frac{n}{2} \right\}$ with poles only for such $s$ that $s(n - s) \in \sigma(\Delta_g)$, and continuous up to
	$\left\{ \re s = n/2 \right\}\setminus \left\{ n/2 \right\}$, such that
	\begin{equation*}
		(\Delta_g + s(n - s))\mathcal{P}(s) = 0
	\end{equation*}
	with expansions
	\begin{align*}
		\mathcal{P}(s)f &= x^{n - s}F + x^sG \text{ if } s \notin n/2 + \mathbb{N}_0/2\\
		\mathcal{P}(s)f &= x^{n/2 - k/2}F + Gx^{n/2 + k/2}\log x \text{ if } s = n/2 + k/2, k \in \mathbb{N},
	\end{align*}
	for $F, G \in C^{\infty}(X)$ such that $F|_{\partial X} = f$.
	If $s = \frac{n}{2} + j$, then $G|_M = -2p_{2k}f$, where $p_{2k}$ is a differential operator on $M$ of order $M$ having principal part $\sigma_{2j}(p_{2j}) = c_j\sigma_{2j}(\Delta_h^j)$, where
	$c_j = (-1)^j[2^{2j}j!(j - 1)!]^{-1}$.
\end{theorem}

With this result in hand, we can define the scattering matrix as an operator $S(s):C^{\infty}(M,|N^*M|^{n - s}) \to C^{\infty}(M,|N^*M|^s)$ for
$\re s \geq n/2, 2s - n \notin \mathbb{N}_0$, and $s(n - s) \notin \sigma(\Delta_g)$. For such $s$, and any $f \in C^{\infty}(M,|N^*M|^{n - s})$, we have by the above
\begin{equation*}
	\mathcal{P}(s)f = x^{n - s}F + x^sG,
\end{equation*}
with $F|_M = f$. The scattering matrix is defined by $S(s)f = G|_M$. It is shown in \cite{gz03} that $S(s)$ extends meromorphically to the entire plane.

The log terms in the theorem arise (when they do arise) for the usual reason seen when the indicial roots of a regular singular ODE are separated by an integer. As the statement makes clear, the log coefficient
may vanish for $\frac{n + q}{2}$ with $q$ odd, but is always nonvanishing for even $q$.

This paper applies the results of \cite{gz03} to the singular Yamabe metric.

\subsection{Notation}
Throughout, $X^{n + 1}$ is a compact manifold with boundary $M^n$ and smooth metric $\bar{g}$. The singular Yamabe metric $g = u^{-2}\bar{g}$ is as above. The distance function to $M$ on $X$ with respect
to $\bar{g}$ is $r$, while $\hat{r}$ is as in Lemma \ref{geodlem}. The induced metric on $M$ is $k = \bar{g}|_{TM}$. When using coordinates, we use the convention that
$r = x^0$, while $x^1,\cdots,x^n$ restrict to coordinates locally on $M$. In index notation, we take $0 \leq i,j \leq n$ and $1 \leq \mu, \nu \leq n$. The second fundamental form
of $M$ with respect to the inward-pointing $\bar{g}$-unit normal $\frac{\partial}{\partial r}$ is denoted by $L$, and the trace-free part by $\mathring{L}$.
The mean curvature of $M$ is $H = k^{\mu\nu}L_{\mu\nu}$. Our curvature sign convention is such that $R_{ij} = R^{k}{}_{ijk}$, and the Laplace operator is a negative operator, i.e.,
the divergence of the gradient.

\section{Local Analysis}
\label{locsec}

In this section we analyze formal solutions to the equation $(\Delta_g + s(n - s))u = 0$ for a singular Yamabe metric $g$. Let $\bar{g}$ be a smooth metric on $X^{n + 1}$, and $u$ the solution to the singular Yamabe equation
(\ref{usyeq}), so that $g = u^{-2}\bar{g}$ has constant scalar curvature $-n(n + 1)$. We again let $M = \partial X$. Near $M$, we write $\bar{g} = dr^2 + h_r$, where $h_r$ is a one-parameter family of metrics on $M$ and $r$
is the $\bar{g}$-distance to $M$. We write $k = h_0 = \bar{g}|_{TM}$.

The following result, which is the primary result of this section, is directly analogous to Proposition 4.2 of \cite{gz03} for the Einstein case, and the proof is modified accordingly. One difference is that
a log term arises here for every integer, whereas in the Einstein cases it arises only for the even integers. The other significant difference is that the metric itself is non-smooth, and via $u$ has logarithmic terms that
must be considered.

\begin{theorem}
	\label{expandprop}
	Let $g$ be the singular Yamabe metric associated to $(X,\bar{g})$, and $f \in C^{\infty}(M)$. For every $q \in \mathbb{N}$ with $1 \leq q \leq n$, and $s = \frac{n + q}{2}$, there is a formal solution $v$ to
	the equation
	\begin{equation*}
		(\Delta_g + s(n - s))v = O(r^{\infty})
	\end{equation*}
	of the form
	\begin{equation*}
		v = r^{\frac{n - q}{2}}(F + Gr^q\log r),
	\end{equation*}
	where $F \in C^{\infty}(X)$, $G \in C^{n - q,1-\varepsilon}(X)$ is polyhomogeneous, and $F|_M = f$. Here $F$ is uniquely determined mod $O(r^q)$ and $G$ is uniquely determined mod $O(r^{\infty})$. In addition,
	\begin{equation}
		\label{gek}
		G|_M = -2c_qP_q^{SY}f,
	\end{equation}
	where $c_q \neq 0$ is a universal constant and $P_q^{SY}$ is a differential operator on $M$ which, if $q$ is even, has principal part $(-\Delta_k)^{\frac{q}{2}}$, and if $q > 1$ is odd,
	has the same principal part as $\mathring{L}^{\mu\nu}\nabla_{\mu}\nabla_{\nu}(\Delta_k)^{\frac{q - 3}{2}}$, where $\mathring{L}$ is the tracefree second fundamental form. If $q = 1$, then
	$G = 0$.

	Finally, $P_q^{SY}$ depends only on the jet of $\bar{g}$ at $M$, and defines a conformally invariant operator $\mathcal{E}(\frac{q - n}{2}) \to \mathcal{E}\left( \frac{-q - n}{2} \right)$.
\end{theorem}
	\begin{proof}[Proof of Theorem \ref{expandprop}]
	We wish to formally solve the equation $(\Delta_g + s(n - s))v = 0$. Now $u = r + O(r^2)$, so we may write $u = r\tilde{u}$ for some $\tilde{u} \in C^n(X)$ satisfying $\tilde{u}|_M = 1$. Thus,
	\begin{align*}
		\Delta_gv &= r^{1 + n}\tu^{1 + n}(\det h)^{-1/2}\partial_i\left[ r^{1 - n}\tu^{1 - n}(\det h)^{1/2}\bar{g}^{ij}\partial_jv \right]\\
		&=r^2\tilde{u}^2\partial_{r}^2v + (1 - n)r\tu^2\partial_rv + (1 - n)r^2\tu\partial_r(\tu)\partial_rv + \frac{1}{2}r^2\tu^2h^{\mu\nu}h'_{\mu\nu}\partial_rv\\
		&\quad+ (1 - n)r^2\tu h^{\mu\nu}\partial_{\mu}(\tu)\partial_{\nu}(v)+ r^2\tu^2\Delta_{h_r}v.
	\end{align*}
	(Here, a prime denotes $\partial_r$.) Taking $v = r^{n - s}\psi$, we find
	\begin{align*}
		[\Delta_g + s(n - s)]v &= r^{n - s + 2}\tu^2\partial_r^2\psi + 2(n - s)r^{n - s + 1}\tu^2\partial_r\psi\\
		&\quad + (n - s)(n - s - 1)r^{n - s}\tu^2\psi + (1 - n)(n - s)r^{n - s}\tu^2\psi\\
		&\quad + (1 - n)r^{n - s + 2}\tu^2\partial_r\psi + (1 - n)r^{n - s + 2}\tu\partial_r(\tu)\partial_r\psi\\
		&\quad + (1 - n)(n - s)r^{n - s + 1}\tu\partial_r(\tu)\psi + \frac{1}{2}r^{n - s + 2}\tu^2h^{\mu\nu}h'_{\mu\nu}\partial_r\psi\\
		&\quad + \frac{1}{2}(n - s)r^{n - s + 1}\tu^2h^{\mu\nu}h'_{\mu\nu}\psi + (1 - n)r^{n - s + 2}\tu h^{\mu\nu}\partial_{\mu}(\tu)\partial_{\nu}(\psi)\\
		&\quad + r^{n - s + 2}\tu^2\Delta_{h_r}\psi + s(n - s)r^{n - s}\psi\\
		&=r^{n - s + 1}\left[ r\tu^2\partial_r^2\psi + \left( (n + 1 - 2s)\tu^2 + (1 - n)r\tu\partial_r\tu\right.\right.\\
		&\quad+ \frac{1}{2}r\tu^2h^{\mu\nu}h'_{\mu\nu} )\partial_r\psi + \left( s(s - n)\frac{\tu^2-1}{r} + (1 - n)(n - s)\tu\partial_r\tu\right.\\
		&\quad+ \frac{1}{2}(n - s)\tu^2h^{\mu\nu}h'_{\mu\nu})\psi + (1 - n)r\tu h^{\mu\nu}\partial_{\mu}\tu\partial_{\nu}\psi\\
	        &\quad\left.+r\tu^2\Delta_{h_r}\psi\right].
	\end{align*}
	We may conclude that
	\begin{equation*}
		[\Delta_g + s(n - s)]\circ r^{n - s} = r^{n - s + 1}\circ \mathcal{D}_s,
	\end{equation*}
	where
	\begin{equation}
		\label{dseq}
		\begin{split}
		\mathcal{D}_s &= r\tu^2\partial_r^2 + \left( (n + 1 - 2s)\tu^2 + (1 - n)r\tu\partial_r\tu + \frac{1}{2}r\tu^2h^{\mu\nu}h'_{\mu\nu}\right)\partial_r\\
		&\quad + s(s - n)\frac{\tu^2 - 1}{r} + (1 - n)(n - s)\tu\partial_r\tu + \frac{1}{2}(n - s)\tu^2h^{\mu\nu}h'_{\mu\nu}\\
		&\quad + (1 - n)r\tu\grad_{h_r}(\tu) + r\tu^2\Delta_{h_r}.
	  	\end{split}
	\end{equation}
	Keeping in mind that $\tu^2 = 1 + O(r)$ and that $\partial_{\mu}\tu = O(r)$, we observe that
	\begin{equation}
		\label{dindeq}
		\mathcal{D}_s(f_jr^j) = j(j + n - 2s)r^{j - 1}f_j + O(r^j).
	\end{equation}
	This equation is the same as that in \cite{gz03}; however, we have avoided expressing the metric in normal form here, since $\bar{g}$ is part of the data of the problem.
	It is also convenient to record that
	\begin{equation}
		\label{dlogeq}
		\mathcal{D}_s(g_jr^j\log r) = (2j + n - 2s)g_jr^{j - 1} + j(j + n - 2s)g_jr^{j - 1}\log(r) + o(r^{j - 1}).
	\end{equation}

	Suppose $n - 2s \notin \mathbb{N}_0$. Then (\ref{dindeq}) allows us to construct a formal solution. Set $f_0 = F_0 = f$. For $j \geq 1$ with $j \leq n$, define $f_j$ by
	\begin{align*}
		j(j + n - 2s)f_j &= -r^{1 - j}\mathcal{D}_s(F_{j - 1})|_{r = 0}\numberthis\label{fjdef}\\
		F_j &= F_{j - 1} + f_jr^j.
	\end{align*}
	Then setting $v_j = r^{n - s}F_j$, we clearly have
	\begin{equation*}
		[\Delta_g + s(n - s)]v_j = O(r^{n - s + j}).
	\end{equation*}
	By induction, we may thus find $v_n = r^{n - s}F_n$ satisfying $[\Delta_g + s(n - s)]v_n = O(r^{2n - s})$. However, examining $\mathcal{D}_s$ in (\ref{dseq}), we see that
	$\mathcal{D}_s(F_n)$ contains a term of the form $a(n - s)f\mathcal{L}r^{n}\log(r)$, for some universal $a \in \mathbb{R}$, via the terms
	$\frac{\tu^2 - 1}{r}$ and $\partial_r\tu$. Here $\mathcal{L}$ is as in (\ref{usyexp}). The induction can therefore be continued only by first adding a term of the form $g_{n + 1}r^{n + 1}\log(r)$ (see (\ref{dlogeq}))
	to cancel the $r^n\log(r)$ term before adding $f_nr^n$. Having done this, the induction can be continued to infinite order, by adding monomials and logarithmic terms, as is standard.
	(In fact, by the polyhomogeneity theorem in \cite{acf92}, higher powers of logarithms may be necessary at high order, but this will be of no concern to us. For a recent very explicit example
	of a construction with this behavior, see \cite{m18a}.)

	By Borel's lemma, this gives us an infinite-order solution $v$. Just as in \cite{gz03}, an easy induction shows that for $j = 2p$ even,
	\begin{equation*}
		f_{2p} = c_{2p,s}P_{2p,s}^{SY}f,\quad c_{2p,s} = (-1)^p\frac{\Gamma\left(s - \frac{n}{2} - p\right)}{2^{2p}p!\Gamma\left(s - \frac{n}{2}\right)},
	\end{equation*}
	where $P_{2p,s}$ is a differential operator on $M$ with principal part equal to that of $(-\Delta_k^p)$.

	The analysis of the leading part of the odd-order terms is somewhat more complicated, because $f_{2p+1}$ contains derivatives only of order $2p$, and there are several different contributions to these.
	Because
	\begin{equation}
		\label{lapderiv}
		\partial_r\Delta_{h_r}|_{r = 0} = \Delta_k\partial_r + 2\mathring{L}^{\mu\nu}\nabla_{\mu}\nabla_{\nu} + \frac{2}{n}H\Delta_k + l.o.t.s
	\end{equation}
	(where $l.o.t.s$ denotes lower-order terms), it is clear from (\ref{dseq}) that
	\begin{equation*}
		f_{2p + 1} = c_{2p + 1,s}\mathring{L}^{\mu\nu}\nabla_{\mu}\nabla_{\nu}\Delta_k^{p - 1}f + d_{2p + 1,s}H\Delta_k^pf + l.o.t.s.
	\end{equation*}
	for some constants $c_{2p + 1,s}$ and $d_{2p + 1,s}$. It is easy to compute, as a base case, that $c_{1,s} = 0$ and $d_{1,s} = \frac{n - s}{2n}$. Now, as
	$\mathring{L}^{\mu\nu}\nabla_{\mu}\nabla_{\nu}\Delta_k^{p - 1}$ and $\Delta_k(\mathring{L}^{\mu\nu}\nabla_{\mu}\nabla_{\nu}\Delta_k^{p - 2})$ have the same principal parts,
	a straightforward induction shows that
	\begin{equation*}
		c_{2p + 1,s} = \frac{1}{2(2p + 1)\left( s - \frac{n}{2} - p - \frac{1}{2} \right)}(2(-1)^{p - 1}c_{2p - 2,s} + c_{2p - 1,s}).
	\end{equation*}
	It is then likewise easy to show by induction that $c_{j,s} > 0$ whenever $s > \frac{n + j}{2}$, for odd $j$.

	We now show by induction that $d_{2p + 1,s} = \frac{n + 2p - 2 - s}{2n}(-1)^pc_{2p,s}$. This is clearly true for $p = 0$. We assume it is true for $p' < p$. Let $j = 2p + 1$.
	We next observe that $\partial_{r}\tu|_{r = 0} = -\frac{1}{2n}H$ (see
	(2.6) in \cite{g17}), while $h^{\mu\nu}h'_{\mu\nu}|_{r = 0} = -2H$. As $\frac{c_{j - 3,s}}{c_{j - 1,s}} = -(j - 1)(2s - n - j)$, we find from (\ref{dseq}) and (\ref{fjdef}) that
	\begin{align*}
		j(2s - n - j)d_{j,s} &= (-1)^pc_{j - 1,s}\left[ \frac{j - 1}{2n}(1 - 2j - 3n + 4s) + \frac{n - s}{2n}(2s - n - 1)\right.\\
		&\quad- \left.\frac{c_{j - 3,s}}{nc_{j - 1,s}} + (-1)^p\frac{d_{j - 2,s}}{c_{j - 1,s}} \right]\\
		&= (-1)^p\frac{c_{j - 1,s}}{2n}\left[ (j - 1)(s - n + 3sn - 2s^2 - n^2 + j(3s - 2n - j))\right.\\
	&\quad+\left. (n - s)(2s - n - 1)\right]\\
	&= (-1)^p\frac{jc_{j - 1,s}}{2n}(j + n - s - 1)(2s - n - j).
	\end{align*}
	This yields the claim regarding $d_{2p + 1,s} = d_{j,s}$. We emphasize that $d_{2p+1,s}$ is smooth across $s = \frac{n + 2p + 1}{2}$.

	For notational consistency, we define $P_{2p + 1,s}^{SY}$ such that $f_{2p + 1} = c_{2p + 1,s}P_{sp + 1,s}f$.

	Now suppose that $2s - n = q$, where $1 \leq q \leq n$, as in the hypothesis. The above construction works until the determination of $f_q$; then the coefficient in (\ref{dindeq}) vanishes, and we
	cannot remove the $r^{q - 1}$ term from $\mathcal{D}_q(F_{q - 1})$. This may be addressed by adding a term of the form $g_qr^q\log(r)$, where $g_q$ is determined by (\ref{dlogeq}) to cancel
	the $r^{q - 1}$ term in $\mathcal{D}_q(F_{q - 1})$. Note that since $q \leq n$, this happens before the logarithm at order $n + 1$ discussed above.
	The expansion then continues as before, with additional logarithmic terms appearing at order $n + 1$ (which limits the smoothness of $G$). The remainder of the claims follow immediately. As in \cite{gz03}, it is clear
	that (\ref{gek}) holds with $P_q^{SY} = P_{q,\frac{n + q}{2}}^{SY}$
	and with $c_q = \res_{s = \frac{n + q}{2}}c_{q,s}$. In particular,
	\begin{equation}
		\label{cqeq}
		\begin{split}
			c_{2p} &= (-1)^p[2^{2p}p!(p - 1)!]^{-1}\\
			c_1 &= 0\\
			c_{2p + 1} &= \frac{1}{2(2p + 1)}\left( 2(-1)^{p - 1}c_{2p - 2,\frac{n + 2p + 1}{2}} + c_{2p - 1,\frac{n + 2p + 1}{2}} \right) > 0.
		\end{split}
	\end{equation}
	Since $d_{2p + 1,s}$ does not contribute to the residue of $P_{2p + 1,s}$ at $s = \frac{n + 2p + 1}{2}$, the principal part of $P_q^{SY}$ is as claimed.
\end{proof}

Note that the conformal invariance property of the operators $P_{q}^{SY}$ can be expressed as follows instead of using bundles: if $\tilde{\bar{g}} = e^{2\omega}\bar{g}$, then
\begin{equation*}
	\widetilde{P}_q^{SY}f = e^{-\frac{(n + q)\omega}{2}}P_q^{SY}\left(e^{\frac{(n - q)\omega}{2}}f\right).
\end{equation*}

\section{Global analysis}\label{globsec}

In section \ref{backsec}, we described the results on the scattering operator and the Poisson operators proved in \cite{gz03}, based on \cite{mm87}. Actually, both
of those papers assume that one is dealing with a smooth AH metric, i.e., a metric whose compactifications are smooth. Although such is not true of AH Einstein metrics
in odd dimension (even boundary dimension), the existence theory for such metrics (given a conformal infinity) is very difficult and largely open, and both existence and
uniqueness are known sometimes to fail. Therefore, papers such as \cite{gz03} and \cite{fg02} made the reasonable decision to restrict attention to smooth \emph{formal} Einstein metrics,
which are Einstein up to as high an order as possible while remaining smooth. The resulting quantities are not well-defined, if global; but given the non-uniqueness of (some) Einstein metrics,
they would not be in any event, in general.

In contrast, the singular Yamabe problem offers both existence and uniqueness for any compact Riemannian manifold with boundary. There is therefore the opportunity to
define genuinely well-defined global quantities if one uses the singular Yamabe metric as the AH metric for the scattering problem. On the other hand, this raises slight technical difficulties:
as we have seen, the singular Yamabe metric is not generally smooth, but has a logarithmic term, and so technically we cannot simply apply results from \cite{gz03} or
\cite{mm87} in a simplistic fashion. We therefore state the following theorem; since the main goals of our paper are geometric, we defer discussing the proof till appendix \ref{proofapp}.

\begin{theorem}
	\label{poissonthm}
	Let $(X^{n + 1},\bar{g})$ be a smooth Riemannian manifold with boundary $M$, let $k = \bar{g}|_{TM}$, and let $g$ be the singular Yamabe metric corresponding to $\bar{g}$. Let $r$ be
	the $\bar{g}$-distance to $M$ on $X$. There is a unique family of Poisson operators
	\begin{equation*}
		\mathcal{P}(s):C^{\infty}(M,|N^*M|^{n - s}) \to C^{\infty}(\mathring{X})
	\end{equation*}
	for $\frac{n}{2} \leq \re s < n + \frac{1}{2}$, which is meromorphic in $\left\{n + \frac{1}{2} > \re s > \frac{n}{2} \right\}$ with poles only for such $s$ that $s(n - s) \in \sigma(\Delta_g)$, and continuous up to
	$\left\{ \re s = n/2 \right\}\setminus \left\{ n/2 \right\}$, such that
	\begin{equation*}
		(\Delta_g + s(n - s))\mathcal{P}(s) = 0
	\end{equation*}
	with expansions
	\begin{align*}
		\mathcal{P}(s) &= r^{n - s}F + r^sG \text{ if } s \notin n/2 + \mathbb{N}_0/2\\
		\mathcal{P}(s)f &= r^{n/2 - q/2}F + Gr^{n/2 + q/2}\log r \text{ if } s = n/2 + q/2, q \in \mathbb{N},
	\end{align*}
	where in the first case, both $F, G \in C^n(X)$ are polyhomogeneous, and in the second case, $F \in C^{\infty}(X)$ and $G \in C^n(X)$ is polyhomogeneous.
	If $s = \frac{n}{2} + j$, then $G|_M = -2p_{2q}f$, where $p_{2q}$ is a differential operator on $M$ of order $M$ having principal part $\sigma_{2j}(p_{2j}) = c_j\sigma_{2j}(\Delta_k^j)$, where
	$c_j$ is as in (\ref{cqeq}).
\end{theorem}
As in the smooth case, we define the scattering operator for $s > \frac{n}{2}$, $s(n - s) \notin \sigma(\Delta_g)$ and $2s - n \notin \mathbb{N}$, by $S(s)f = G|_M$, where
$\mathcal{P}(s)f = r^{n - s}F + r^sG$ as above. We obtain the following (see appendix \ref{proofapp}).
\begin{theorem}
	\label{scatterthm}
	The scattering operator $S(s)$ has a meromorphic extension to the strip $-\frac{1}{2} < s < n + \frac{1}{2}$, regular for $\re s = \frac{n}{2}$. On that strip, Propositions 3.6 - 3.10 of
	\cite{gz03} remain true.
\end{theorem}
The reason for the restriction on $s$ in both of these statements is the following: as seen in the proof of Proposition \ref{expandprop}, the logarithmic term in the expansion of the singular Yamabe function
$u$ implies that a term of the form $r^{2n + 1 - s}\log(r)$ will appear in the asymptotic expansion of $\mathcal{P}(s)f$. In order to be able to ignore this phenomenon in analyzing the scattering operator,
we require that the higher indicial root, $s$, occur before this power, i.e., $s < 2n + 1 - s$. The remaining cases could be analyzed, but they would be rather more involved, and in any event are not necessary
for what we wish to study here.

We are ready to prove Theorem \ref{PSYthm}.
\begin{proof}[Proof of Theorem \ref{PSYthm}]
	The existence and meromorphicity of the scattering operator is Theorem \ref{scatterthm}. Theorem \ref{expandprop}, together with Proposition 3.6 of \cite{gz03}, gives the remainder of the claims.
\end{proof}

Having defined the scattering operator, we can also finally define our $\syQ$-curvature as in the introduction:
\begin{equation}
	\label{syqdef}
	\syQ = c_{n}^{-1}S(n)1,
\end{equation}
where $c_q$ is as in (\ref{cqeq}). Note that although $S(s)$ actually has a pole at $s = n$, the residue has no constant term, so that $S(s)1$ continues holomorphically across $s = n$. This is discussed in general
in \cite{gz03}.

We also can now define the ``fractional order GJMS operators'' for this setting, following \cite{cg11} (where in fact this setting was mentioned, but not discussed in detail). We define
$P_{2\gamma}^{SY} = S(\frac{n}{2} + \gamma)$. Once again this notation differs by a factor of two from the notation in \cite{cg11}, which is done in order to be consistent with the integer case. For us,
$P_{2\gamma}$ can be viewed as a pseudodifferential operator of order $2\gamma$; see p. 107 of \cite{gz03}.

We next prove the following theorem. The proof in the Einstein case is given in \cite{fg02}, and in fact the same proof works here. To refresh the reader's memory and demonstrate that the proof
carries over, and because we will want to refer back to it, we reproduce the proof here.
\begin{theorem}
	\label{globuthm}
	Let $g$ be the singular Yamabe metric on $(X^{n + 1},\bar{g})$, as above; and let $r$ be the $\bar{g}$-distance function to $M^n = \partial X$. There is a unique function $U \in C^{\infty}(\mathring{X})$
	solving
	\begin{equation*}
		-\Delta_gU = n,
	\end{equation*}
	with asymptotics
	\begin{equation*}
		U = \log(r)  + A + Br^n\log r,
	\end{equation*}
	where $A \in C^{\infty}(X)$ and $B \in C^n(X)$, and where $A|_M = 0$ and $B|_M = -2c_n\syQ$.
\end{theorem}
\begin{proof}
	By uniqueness, we have $\mathcal{P}(n)1 = 1$. Now for all $s$ near $n$, we of course have
	\begin{equation}
		\label{theproofeq}
		(\Delta_g + s(n - s))\mathcal{P}(s)1 = 0,
	\end{equation}
	and $\mathcal{P}(s)1$ is a holomorphic family of functions on $\mathring{X}$.

	Set $U = -\frac{d}{ds}\mathcal{P}(s)|_{s = n}$. Then differentiating (\ref{theproofeq}) with respect to $s$ and taking $s = n$ gives
	\begin{equation*}
		\Delta_gU = -n.
	\end{equation*}
	On the other hand, near $s = n$ we can write
	\begin{equation}
		\label{theproofeq2}
		\mathcal{P}(s)1 = r^{n - s}F_s + r^sG_s.
	\end{equation}
	This expansion can be extended to $s = n$, where $\mathcal{P}(n)1 = 1$, and it was shown in \cite{gz03} and pointed out in \cite{fg02} that the extension is unique subject to the requirement that
	$F_s, G_s$ be chosen to depend holomorphically on $s$ across $s = n$. (Actually, in the Einstein case considered in \cite{fg02}, this argument is needed only for $n$ even, due to evenness properties
	available in that case. In our situation, it applies to both even and odd $n$.) 

	It follows from the definition (\ref{syqdef}) of $\syQ$ and the definition of the scattering matrix that $G_n|_M = c_{n}\syQ$; thus,
	$F_n = 1 - c_{n}\syQ r^n$.
	
	We can differentiate both sides of (\ref{theproofeq2}) and conclude
	\begin{equation}
		\label{uexpeq}
		U = F_n\log(r) - F_n' - G_nr^nlog(r) - r^nG_n',
	\end{equation}
	with $' = \frac{d}{ds}$. Since $F_s|_M \equiv 1$, the second term vanishes. The result follows from the above asymptotics.
\end{proof}

Precisely the same arguments given in section 3 of \cite{fg02} for the even-$n$ Einstein case now immediately work, given Proposition \ref{expandprop} and Theorem \ref{globuthm}, to produce Theorem \ref{thmQ} and
Corollary \ref{corQ}. We do not reproduce those proofs here.

Note that it follows from Corollary \ref{corQ} and from \cite{g17} that the boundary integral of $\syQ$ is a global conformal invariant. Furthermore, two consequences of the transformation rule in
Theorem \ref{thmQ} is that $\syQ$ is generically nonzero (since $P_n^{SY}$ is generically nontrivial); and, although $\syQ$ is extrinsic in the sense that it depends on $\bar{g}$ and not only
$k = \bar{g}|_{TM}$, it is nevertheless the case that \emph{given a conformal class} $[\bar{g}]$, $\syQ$ depends only on $\bar{g}|_{TM}$.

We may now prove Theorem \ref{globthm1} using an argument from \cite{cqy08}.
\begin{proof}[Proof of Theorem \ref{globthm1}]
	Let $U$ be as in theorem \ref{globuthm}. Then by Green's theorem, we have
	\begin{equation}
		\label{volinteq}
		\vol_g(\left\{ r > \varepsilon \right\}) = -\frac{1}{n}\int_{r > \varepsilon}\Delta_gUdv_g = -\frac{1}{n}\int_{r = \varepsilon}\frac{\partial U}{\partial n}dv_{h_{\varepsilon}} =
		\frac{1}{n}\varepsilon^{1 - n}\oint_M\left.\frac{\partial U}{\partial r}\right|_{r = \varepsilon}dv_{h_{\varepsilon}},
	\end{equation}
	since the outward normal to $\left\{ x = \varepsilon \right\}$ is $-\varepsilon\frac{\partial}{\partial r}$.
	Now, by (\ref{uexpeq}),
	\begin{equation*}
		\begin{split}
			\frac{\partial U}{\partial r} = &\frac{1}{r} - \cdots - (n - 1)a_{n-1}'r^{n - 1} - na_n'r^n - n\left( \left.\frac{d}{ds}\right|_{s = n}S(s)1 \right)r^{n - 1}\\
			&\quad-2c_{n}\syQ r^{n - 1} - 2nc_{n}\syQ r^{n - 1}\log(r) + o(r^{n - 1}).
		\end{split}
	\end{equation*}
	Since by \cite{g17} $\oint_Mv^{(n)}dv_k = \mathcal{E} = 2c_{n}\oint_M\syQ dv_k$,
	the result follows from (\ref{volinteq}) by collecting zeroth-order terms in $\varepsilon$.
\end{proof}

Finally, with the above pieces all in place, the proof of Theorem \ref{globthm2} is identical to that in \cite{cqy08}.

\section{Explicit Computations}
\label{compsec}

In this section, we compute several of the quantities and operators defined earlier for low dimensions. Throughout, we work near the boundary $M^n = \partial X^{n + 1}$, on a collar neighborhood $V$
with a diffeomorphism $\psi:[0,\varepsilon)_r \times M \to V$ such that $\psi^*\bar{g} = dr^2 + h_r$. We use Greek indices $1 \leq \mu,\nu \leq n$ on $M$, and Latin indices $0 \leq i, j \leq n$ on $X$.
In particular, $x^0 = r$. We set $k = h_0 = \bar{g}|_{TM}$.

We begin with the following lemma, which is an expansion of the metric in Fermi coordinates. This is
well-known to first order. The result to this order is contained in \cite{gg19}.

\begin{lemma}
	\label{fermilem}
	Let $X^{n + 1}$ be a smooth manifold with boundary $\partial X = M^n$, and $\bar{g}$ a smooth metric on $X$. Suppose that, near $M$, the metric $g$ is written as
	$\bar{g} = dr^2 + h_r$, where $h_r$ is a one-parameter family of metrics on $M$. Let $k = \bar{g}|_{TM}$. Then
	\begin{equation*}
		h_{\mu\nu} = k_{\mu\nu} - 2rL_{\mu\nu} + r^2(L_{\mu}^{\sigma}L_{\nu\sigma} - \overline{R}_{0\mu\nu 0}) - \frac{1}{3}r^3(\overline{\nabla}_0\overline{R}_{0\mu\nu 0}-4 L^{\sigma}_{(\mu}\overline{R}_{\nu)00\sigma})
		+ O(r^4),
	\end{equation*}
	where $L$ is the second fundamental form of $M$ with respect to $\bar{g}$ and $\overline{R}$ is the curvature tensor of $\bar{g}$. Moreover, all coefficients are evaluated at $r = 0$, and indices raised
	with $k^{-1}$.
\end{lemma}

To compute the local invariants $\syQ$ and $P^{SY}_k$, we must formally solve the scattering problem $(\Delta_g + s(n - s))u = 0$ for $g$ the singular Yamabe metric. The straightforward approach to this
entails computing $u$ from $\bar{g}$, then computing the scattering operator of $g = u^{-2}\bar{g}$, and finally formally solving that equation. This approach is extremely tedious, and we will pursue
a different one. Let $\hat{r}$ be a geodesic defining function with respect to $g = u^{-2}\bar{g}$, of the kind constructed in Lemma \ref{geodlem}, and such that
$\frac{\partial \hat{r}}{\partial r}|_{M} = 1$. Then for $\hat{\bar{g}} = \hat{r}^2g$, which is a
$C^n$ compact metric on $X$, the singular Yamabe function $\hat{u}$ is equal to $\hat{r}$, the $\hat{\bar{g}}$-distance to $M$. This drastically simplifies the computation of
$\hat{\syQ}$ and $\hat{P}^{SY}_k$. Morever, because (within a conformal class) both depend only on the boundary representative, and because $\hat{\bar{g}}|_{TM} = \bar{g}|_{TM}$ by construction,
in fact we have in this way computed $\syQ$ and $P_k^{SY}$. All that remains is to express the computed quantity in terms of invariants of our original metric $\bar{g}$, and for that task, the following
three lemmas are the tools we will need.

The first is completely standard, and is included here only for convenience. The proof is left as an exercise.
\begin{lemma}
	\label{confchangelem}
	Suppose $\bar{g}$ is a smooth metric on $X^{n + 1}$, and that $\tilde{\bar{g}} = e^{2\omega}\bar{g}$. Let $k = \bar{g}|_{TM}$. Then extrinsic invariants at the boundary $M$ transform as follows:
	\begin{align*}
		\widetilde{H} &= e^{-\omega}(H - n\omega_r)\\
		\mathring{\widetilde{L}}_{\mu\nu} &= e^{\omega}\mathring{L}_{\mu\nu}\\
		\widetilde{\overline{R}}_{\mu\nu} &= \overline{R}_{\mu\nu} - (n - 1)\overline{\nabla}^{2}_{\mu\nu}\omega + (n - 1)\omega_{\mu}\omega_{\nu} - (\Delta_{\bar{g}}\omega - (n - 1)|d\omega|_{\bar{g}}^2)\bar{g}_{\mu\nu}\\
		\widetilde{R}_{\mu\nu} &= R_{\mu\nu} - (n - 2)\nabla^2_{\mu\nu}\omega + (n - 2)\omega_{\mu}\omega_{\nu} - (\Delta_k\omega + (n - 2)|d\omega|_k^{2})k_{\mu\nu}\\
		\widetilde{\overline{R}} &= e^{-2\omega}(\overline{R} - 2n\Delta_{\bar{g}}\omega - n(n - 1)|d\omega|_{\bar{g}}^2)\\
		\widetilde{R} &= e^{-2\omega}(R - 2(n - 1)\Delta_k\omega - (n - 1)(n - 2)|d\omega|_k^2).
	\end{align*}
	Here $\overline{R}_{\mu\nu}$ and $\overline{R}$ are the Ricci and scalar curvatures for $\bar{g}$, and $R_{\mu\nu}$ and $R$ are the Ricci and scalar curvatures for $k$; and similarly for 
	$\widetilde{\overline{R}}$, etc. Moreover, $|d\omega|_{\bar{g}}^2$ is the squared $\bar{g}$-norm of $d\omega$, where $d$ is the exterior derivative on $X$; while $|d\omega|_k^2$ is the squared $k$-norm of
	$d(\omega|_M)$, where $d$ is the exterior derivative on $M$. Finally, $L$ is the second fundamental form of $M$ with respect to the inward-pointing normal $\frac{\partial}{\partial r}$, and
	$H = k^{\mu\nu}L_{\mu\nu}$.
\end{lemma}

\begin{lemma}
	\label{uexplem}
	Let $\bar{g}$ be a smooth metric on $X^{n + 1}$, and let $u$ be the singular Yamabe solution for $\bar{g}$, so that $g = u^{-2}\bar{g}$ has constant scalar curvature $-n(n + 1)$.
	Then $u = r\tilde{u}$, where $\tilde{u} = 1 + O(r)$. Moreover,
	\begin{align*}
		\partial_r\tu|_{r = 0} &= -\frac{1}{2n}H\text{ and}\\
		\partial_r^2\tu|_{r = 0} &= -\frac{1}{3n}(\overline{R} + H^2) + \frac{1}{3(n - 1)}(R - |\mathring{L}|_k^2).
	\end{align*}
\end{lemma}
This is proved in \cite{g17}. See in particular equation (2.6) and pages 1788-89.

\begin{lemma}
	\label{omegalem}
	Let $\bar{g}$ be a smooth metric on $X^{n + 1}$, and let $g$ be the corresponding singular Yamabe metric. Let $\hat{r}$ be the geodesic defining function for $M$ (with respect to $g$), as in Lemma
	\ref{geodlem}, such that $\frac{\partial \hat{r}}{\partial r}|_{M} = 1$. Then if $\hat{\bar{g}} = r^2g$ is the corresponding geodesic compactification of $g$, we can write
	$\hat{\bar{g}} = e^{2\omega}\bar{g}$, where
	\begin{align*}
		\omega|_{M} &= 0\\
		\partial_r\omega|_{M} &= \frac{1}{n}H\\
		\partial_r^2\omega|_M &= \frac{1 + n}{2n^2}H^2 + \frac{1}{2n}\overline{R} - \frac{1}{2(n - 1)}R + \frac{1}{2(n - 1)}|\mathring{L}|_k^2\\
		\partial_r^3\omega|_M &= -\frac{1}{n}\Delta_kH + \frac{1}{n - 2}\nabla^{\mu}\nabla^{\nu}\mathring{L}_{\mu\nu} + \frac{1}{n - 2}\mathring{L}^{\mu\nu}\overline{R}_{\mu\nu} - \frac{2}{n - 2}
		\mathring{L}^{\mu\nu}R_{\mu\nu}\\
		&\quad+\frac{1}{2n}\partial_r\overline{R} + \frac{n^2 + 2n + 1}{2n^3}H^3 + \frac{n + 1}{2n^2}H\overline{R} - \frac{n + 2}{2n(n - 1)}HR\\
		&\quad+ \frac{3n^2 - 4n - 2}{2n(n - 1)(n - 20)}H|\mathring{L}|_k^2.
	\end{align*}
\end{lemma}
\begin{proof}
	Observe that the singular Yamabe function $u = r\tu$ is a defining function for $M$. Thus, taking $r_0 = u$ in equation (2.2) of (\cite{g99}) gives $\hat{\bar{g}} = e^{2\omega}\bar{g}$ where
	\begin{equation*}
		2(\grad_{\bar{g}}u)(\omega) + u|d\omega|_{\bar{g}}^2 = \frac{1 - |du|_{\bar{g}}^2}{u}.
	\end{equation*}
	We can write this equation as
	\begin{equation*}
		2r^2\tu\bar{g}^{ij}\partial_i\tu\partial_j\omega + 2r\tu^2\partial_r\omega + r^2\tu^2\bar{g}^{ij}\partial_i\omega\partial_j\omega = 1 - \tu^2 - 2r\tu\partial\tu - r^2\bar{g}^{ij}\partial_i\tu\partial_j\tu.
	\end{equation*}
	Now, tangential derivatives of both $\tu$ and $\omega$ vanish to first order, so we can rewrite this as
	\begin{equation*}
		2r^2\tu\partial_r\tu\partial_r\omega + 2r\tu^2\partial_r\omega + r^2\tu^2\partial_r(\omega)^2 = 1 - \tu^2 - 2r\tu\partial_r\tu - r^2\partial_r(\tu)^2 + O(r^4).
	\end{equation*}
	Differentiating gives
	\begin{equation*}
		\begin{split}
			8r\tu\partial_r(\tu)\partial_r(\omega) +&2r^2\partial_r(\tu)^2 + 2r^2\tu\partial_r^2(\tu)\partial_r(\omega) + 2r^2\tu\partial_r(\tu)\partial_r^2(\omega)\\
			&+2\tu^2\partial_r(\omega) + 2r\tu^2\partial_r^2(\omega) + 2r\tu^2\partial_r(\omega)^2 + 2r^2\tu\partial_r(\tu)\partial_r(\omega)^2\\
			&+2r^2\tu^2\partial_r(\omega)\partial_r^2(\omega) = -4\tu\partial_r(\tu) - 4r\partial(\tu)^2 - 2r\tu\partial_r^2(\tu)\\
			&-2r^2\partial_r(\tu)\partial_r^2(\tu) + O(r^3).
		\end{split}
	\end{equation*}
	Taking $r = 0$ and applying Lemma \ref{uexplem} gives
	\begin{equation*}
		\partial_r\omega|_{r = 0} = \frac{1}{n}H.
	\end{equation*}
	Differentiating again mod $O(r^2)$, we find
	\begin{equation*}
	\begin{split}
		12\tu\partial_r(\tu)\partial_r(\omega) +& 12r\partial_r(\tu)^2\partial_r(\omega) + 12r\tu\partial_r^2(\tu)\partial_r(\omega) + 12r\tu\partial_r(\tu)\partial_r^2(\omega)\\
		&+ 4\tu\partial_r(\tu)\partial_r(\omega) + 4\tu^2\partial_r^{2}(\omega) + 4r\tu\partial_r(\tu)\partial_r^2(\omega)\\
		&+ 2r\tu^2\partial_r^3(\omega) + 2\tu^2\partial_r(\omega)^2 + 8r\tu\partial_r(\tu)\partial_r(\omega)^2\\
		&+ 8r\tu^2\partial_r(\omega)\partial_r^2(\omega) = -8\partial_r(\tu)^2 - 6\tu\partial_r^2\tu - 14\partial_r(\tu)\partial_r^2(\tu)\\
		&\quad- 2r\tu\partial_r^3(\tu) + O(r^2).
	\end{split}
	\end{equation*}
	and setting $r = 0$ and using Lemma \ref{uexplem} gives
	\begin{equation*}
		\partial_r^2\omega|_{r = 0} = \frac{1-n}{2n}H^2 + \frac{1}{2n}\overline{R} - \frac{1}{2(n - 1)}R + \frac{1}{2(n - 1)}|\mathring{L}|_k^2.
	\end{equation*}

	Finally, we differentiate again, mod $O(r)$:
	\begin{equation*}
		\begin{split}
			24\partial_r(\tu)^2\partial_r(\omega) +& 24\tu\partial_r^2(\tu)\partial_r(\omega) + 36\tu\partial_r(\tu)\partial_r^2(\omega) + 6\tu^2\partial_r^3(\omega)\\
			&+12\tu\partial_r(\tu)\partial_r(\omega)^2 + 12\tu^2\partial_r(\omega)\partial_r^2(\omega) = -30\partial_r(\tu)\partial_r^2(\tu)\\
			&-6\partial_r(\tu)\partial_r^2(\tu)-8\tu\partial_r^2(\tu) + O(r).
		\end{split}
	\end{equation*}
	Taking $r = 0$ and using our prior results gives the claimed formula for $\partial_r^3\omega|_{M}$.
\end{proof}

We also record some additional elementary formulas. First, it follows from Gauss's equation that
\begin{equation}
	\label{LRconteq}
	\mathring{L}^{\mu\nu}\overline{R}_{0\mu\nu 0} = \mathring{L}^{\mu\nu}\overline{R}_{\mu\nu} - \mathring{L}^{\mu\nu}R_{\mu\nu} + \frac{n - 2}{n}H|\mathring{L}|_k^2 - \mathring{L}^3,
\end{equation}
where $\mathring{L}^3 = \mathring{L}_{\alpha}^{\beta}\mathring{L}_{\beta}^{\mu}\mathring{L}_{\mu}^{\alpha}$.
Next, it is easy to show using Codazzi's equation that
\begin{equation}
	\label{codazeq}
	\begin{split}
		\overline{\nabla}_r\overline{R}_{00}|_{r = 0} =& \frac{1}{2}\partial_r\overline{R} + \frac{1 - n}{n}\Delta_kH + \nabla^{\mu}\nabla^{\nu}\mathring{L}_{\mu\nu} - \mathring{L}^{\mu\nu}\overline{R}_{\mu\nu}\\
		&\quad+\frac{n - 1}{2n}H\overline{R} - \frac{1 + n}{2n}HR - \frac{1 + n}{2n}H|\mathring{L}|_k^2 + \frac{n^2 - 1}{2n^2}H^3.
  	\end{split}
\end{equation}
Here, $\nabla$ is the Levi-Civita connection of $k$.

We now begin to analyze the singular Yamabe solution.

As mentioned above, we let $g$ be the singular Yamabe metric corresponding to $\bar{g}$ on $X^{n + 1}$, and let $\hat{r}$ be the geodesic defining function corresponding to $k$, with $\hat{\bar{g}} = \hat{r}^2g$.
That is, $g = \frac{d\hr^2 + \hat{h}_{\hr}}{\hr^2}$, where $\hat{h}_{\hr}$ is a one-parameter family of smooth metric on $M$ with $\hat{h}_0 = k$.
By Lemma \ref{geodlem}, $\hr$ is $C^{n}$, and thus so is $\hat{\bar{g}}$. Also $r\hat{\bar{g}}$ is $C^{n + 1}$. Now, with respect to $\hat{\bar{g}}$, the singular Yamabe function $\hat{u}$ one obtains by starting with
$\hat{\overline{g}}$ is $\hat{r}$ itself; this follows because
$g = \hat{r}^{-2}\hat{\bar{g}}$, and the singular Yamabe function is unique.

We introduce some further notations related to $\hat{\bar{g}}$. The Ricci and scalar curvatures we will denote by $\overline{R}_{ij}$ and $\overline{R}$, respectively. The second fundamental form with respect to
$\hat{\bar{g}}$ we will denote, for convenience, by $\hat{L}$. The tracefree part of this we will call $\mathring{L}$, since it is a conformal invariant at the boundary and thus identical to the corresponding tensor for $\bar{g}$,
since $\bar{g}|_{TM} = \hat{\bar{g}}|_{TM} = k$. We use $\hat{H}$ for the mean curvature $k^{\mu\nu}\hat{L}_{\mu\nu}$.
Meanwhile, we will continue to use $R$ to denote the intrinsic curvature $R_k$ of the boundary metric $k$.

Now, the equation satisfied by $\hat{u}$, i.e. the singular Yamabe equation, is given by
\begin{equation*}
	n(n + 1) = n(n + 1)|du|_{\hat{\bar{g}}}^2 - 2n\hat{u}\Delta_{\hat{\bar{g}}}\hat{u} - \hat{u}^2\hat{\overline{R}}.
\end{equation*}
(See (\ref{usyeq}).) Since $\hat{u} = \hat{r}$, we have
\begin{equation*}
	0 = 2n\Delta_{\hat{\bar{g}}}\hat{r} + \hat{r}\hat{\overline{R}}.
\end{equation*}
Since $\Delta_{\hat{\bar{g}}}\hr = \frac{1}{2}\hat{h}^{\mu\nu}\hat{h}'_{\mu\nu}$, this can be re-expressed as
\begin{equation*}
	\hat{h}^{\mu\nu}\hat{h}'_{\mu\nu} = -\frac{1}{n}\hat{r}\hat{\overline{R}}.
\end{equation*}
Taking $\hr = 0$ and applying Lemma \ref{fermilem} immediately gives 
\begin{equation}
	\label{n3geo1eq}
	\hat{H} = 0.
\end{equation}
This implies also that $\hat{L} = \mathring{L}$.
We differentiate, using the identity $(h^{\mu\nu})' = -h^{\mu\alpha}h^{\nu\beta}h'_{\alpha\beta}$, to find
\begin{equation}
	\label{twicediffgeod}
	-\hh^{\mu\alpha}\hh^{\nu\beta}\hh'_{\alpha\beta}\hh'_{\mu\nu} + \hh^{\mu\nu}\hh''_{\mu\nu} = -\frac{1}{n}\hat{\overline{R}} - \frac{1}{n}r\partial_{\hr}\hat{\overline{R}}.
\end{equation}
Taking $\hr = 0$ and again using Lemma \ref{fermilem} gives
\begin{equation}
	\label{firsdiff}
	-4|\hat{L}|^2_k - 2\hat{\overline{R}}_{00} + 2|\mathring{L}|_k^2 = -\frac{1}{n}\hat{\overline{R}}.
\end{equation}
(Recall that $\hat{L}$ is the second fundamental form.) Now, it is easy to show, for any metric $\bar{g}$ on an (n + 1)-dimensional manifold, that
\begin{equation*}
	|\hat{L}|_k^2 = |\mathring{L}|_k^2 + \frac{1}{n}\hat{H}^2.
\end{equation*}
Similarly, we have by Gauss's formula that
\begin{equation*}
	\hat{\overline{R}}_{00} = \frac{1}{2}\left( \hat{\overline{R}} - R - |\mathring{L}|_k^2 + \frac{n - 1}{n}\hat{H}^2 \right)
\end{equation*}
(see equation (4.3) in \cite{g17}, keeping in mind that that paper has a different convention for curvature indices).
Applying these equations to $\hat{\bar{g}}$, and recalling that $\hat{H} = 0$, we conclude from (\ref{firsdiff}) that at $r = 0$,
\begin{equation}
	\label{n3geo2eq}
	\hat{\overline{R}} = \frac{n}{n - 1}(R - |\mathring{L}|_k^2).
\end{equation}
Note that this result is good for $n \geq 2$, by Lemma \ref{geodlem}.
We differentiate (\ref{twicediffgeod}) again, which we can do if $n \geq 3$ according to Lemma \ref{geodlem}. We find
\begin{equation*}
	2\hh^{\mu\sigma}\hh^{\alpha\lambda}\hh'_{\sigma\lambda}\hh'_{\alpha\beta}\hh'_{\mu\nu} - 3\hh^{\mu\alpha}\hh^{\nu\beta}\hh'_{\alpha\beta}\hh_{\mu\nu}'' + \hh^{\mu\nu}\hh'''_{\mu\nu}
	= -\frac{2}{n}\partial_{\hat{r}}\hat{\overline{R}} - \frac{1}{n}\hat{r}\partial_{\hat{r}}^2\hat{\overline{R}}.
\end{equation*}
Taking $\hr = 0$ and applying Lemma \ref{fermilem}, we get
\begin{equation*}
	-4\hat{L}^3 - 4\hat{L}^{\mu\nu}\hat{\overline{R}}_{0\mu\nu 0}-2\hat{\overline{\nabla}}_0\hat{\overline{R}}_{00} = -\frac{2}{n}\partial_{\hr}\hat{\overline{R}}.
\end{equation*}
Now using our previous calculations, (\ref{LRconteq}) and (\ref{codazeq}), we find
\begin{equation}
	\label{n3geo3eq}
	\partial_{\hr}\hat{\overline{R}}|_{\hr = 0} = -\frac{2n}{n - 2}\nabla^{\mu}\nabla^{\nu}\mathring{L}_{\mu\nu} + \frac{4n}{n - 2}\mathring{L}^{\mu\nu}R_{\mu\nu} - \frac{2n}{n - 2}
	\mathring{L}^{\mu\nu}\hat{\overline{R}}_{\mu\nu}.
\end{equation}
Here, $\nabla$ and $R$ are the connection and Ricci curvature, respectively, of $k$. We leave hats off $\nabla, R$, and $\mathring{L}$ because these are the same for $\hat{\bar{g}}$ as for $\bar{g}$.

Having derived the geometric consequences of using a geodesic compactification, we next compute $\Delta_g$ for any $v \in C^{\infty}(X)$, where (recall) $g$ is the singular Yamabe metric. We find
\begin{align*}
	\Delta_gv &= (\det g)^{-1/2}\partial_i\left[ (\det g)^{1/2}g^{ij}\partial_jv \right]\\
	&= \hr^{1 + n}(\det \hat{\bar{g}})^{-1/2}\partial_i\left[ \hr^{1 - n}(\det \hat{\bar{g}})^{1/2}\hat{\bar{g}}^{ij}\partial_jv \right]\\
	&= \hr^2\partial_{\hr}^2v + (1 - n) \hr\partial_{\hr}v + \frac{1}{2}\hr^2\hh^{\mu\nu}\hh'_{\mu\nu}\partial_{\hr}v + \hr^2\Delta_{\hh_{\hr}}v.
\end{align*}
Now if $f \in C^{\infty}(M)$ and $v = \hr^{n - s}f$, we therefore find
\begin{equation}
	\label{baselinedelta}
	\Delta_gv = s(s - n)\hr^{n - s}f + \frac{n - s}{2}\hr^{n - s + 1}\hh^{\mu\nu}\hh'_{\mu\nu}f + \hr^{n - s + 2}\Delta_{\hh_{\hr}}f.
\end{equation}
We expand the various quantities that appear in this expression. First, a straightforward computation shows that
\begin{equation}
	\label{Deltaeq}
	\Delta_{\hh_{\hr}}f = \Delta_kf + 2\hr\left[ \mathring{L}^{\mu\nu}\nabla^2_{\mu\nu}f + \nabla^{\mu}\mathring{L}_{\mu}^{\nu}f_{\nu} \right] + O(\hr^2).
\end{equation}
Meanwhile, by iteratively applying the equation $(\hh^{\mu\nu})' = -\hh^{\mu\alpha}\hh^{\nu\beta}\hh'_{\alpha\beta}$ along with Lemma \ref{fermilem} and equations (\ref{n3geo1eq}) - (\ref{n3geo3eq}), we find
\begin{equation*}
	\hh^{\mu\nu}\hh'_{\mu\nu} = \frac{1}{n - 1}\hr(|\mathring{L}|_k^2 - R) + \frac{2}{n - 2}\hr^2(\nabla^{\mu}\nabla^{\nu}\mathring{L}_{\mu\nu} - 2\mathring{L}^{\mu\nu}R_{\mu\nu}
	+\mathring{L}^{\mu\nu}\hat{\overline{R}}_{\mu\nu}).
\end{equation*}
Thus, we conclude from (\ref{baselinedelta}) and (\ref{Deltaeq}) that
\begin{equation}
	\label{n3opeq}
	\begin{split}
		[\Delta_g + s(n - s)](\hr^{n - s}f) =& \hr^{n - s + 2}\left[ \frac{n - s}{2(n - 1)}(|\mathring{L}|_k^2 - R) + \Delta_kf \right]\\
		&+ \hr^{n  - s + 3}\left[ \frac{n - s}{n - 2}(\nabla^{\mu}\nabla^{\nu}\mathring{L}_{\mu\nu} - 2\mathring{L}^{\mu\nu}R_{\mu\nu} + \mathring{L}^{\mu\nu}\hat{\overline{R}}_{\mu\nu})\right.\\
		&+ \left.2\mathring{L}^{\mu\nu}\nabla^2_{\mu\nu}f + 2\nabla^{\mu}\mathring{L}_{\mu}^{\nu}f_{\nu}\right] + O(\hr^{n - s + 4}).
	\end{split}
\end{equation}

We now formally solve the equation $(\Delta_g + s(n - s))v = 0$. Let $f \in C^{\infty}(M)$ be arbitrary, and set $v_0 = r^{n - s}f$. We will perturb $v_0$ at increasing orders to formally solve the equation.
To do this, we compute the indicial operator $I_s^j:C^{\infty}(M) \to C^{\infty}(M)$, which we define by
\begin{equation*}
	I_s^j(\psi) = \hr^{-(n - s + j)}[\Delta_g + s(n - s)](r^{n - s + j}\psi)|_{\hr = 0}.
\end{equation*}
This operator tells us the effect of a perturbation of $v$ at order $\hr^{n - s + j}$ on $[\Delta_g + s(n - s)]v$.
It is easy to compute from (\ref{baselinedelta}) that
\begin{equation*}
	I_s^j(\psi) = j(n - 2s + j)\psi.
\end{equation*}
Now it follows from (\ref{n3opeq}) that $[\Delta_g + s(n - s)]v_0 = O(\hr^{n - s + 2})$, so we want to perturb at order $\hr^{n - s + 2}$.
Specifically, we wish to solve
\begin{equation*}
	I_s^2\psi_2 = -\left(\frac{n - s}{2(n - 1)}(|\mathring{L}|_k^2 - R)\right)f + \Delta_kf,
\end{equation*}
which gives
\begin{equation*}
	\psi_2 = \frac{1}{(n - 1)(n + 2 - 2s)}\frac{n - s}{4}(R - |\mathring{L}|_k^2)f - \frac{1}{2(n + 2 - 2s)}\Delta_kf.
\end{equation*}
We therefore set $v_2 = \hr^{n - s}(f + \hr^2\psi_2)$. (We skip $v_1$ in our numbering since there is no term of order $\hr^{n - s + 1}$ in (\ref{n3opeq}).)
It can easily be shown from (\ref{n3opeq}) that, apart from removing the order $\hr^{n - s + 2}$ term from $(\Delta_g + s(n - s))v_0$, adding this perturbation
to $v_0$ has no other effects before order $\hr^{n - s + 4}$. Thus, the next equation we wish to solve is
\begin{equation*}
	\begin{split}
		I_s^3\psi_3 =& \frac{s - n}{n - 2}(\nabla^{\mu}\nabla^{\nu}\mathring{L}_{\mu\nu} - 2\mathring{L}^{\mu\nu}R_{\mu\nu} + \mathring{L}^{\mu\nu}\hat{\overline{R}}_{\mu\nu})\\
		&-2\mathring{L}^{\mu\nu}\nabla^2_{\mu\nu}f - 2\nabla^{\mu}\mathring{L}_{\mu}^{\nu}f_{\nu}.
	\end{split}
\end{equation*}
Since $I_s^3 = 3(n + 3 - 2s)$, we obtain
\begin{equation}
	\label{psi3}
	\begin{split}
		\psi_3 =& \frac{n - s}{3(n - 2)(n + 3 - 2s)}(-\nabla^{\mu}\nabla^{\nu}\mathring{L}_{\mu\nu} + 2\mathring{L}^{\mu\nu}R_{\mu\nu} - \mathring{L}^{\mu\nu}\hat{\overline{R}}_{\mu\nu})f\\
		& - \frac{2}{3(n + 3 - 2s)}(\mathring{L}^{\mu\nu}\nabla^2_{\mu\nu}f + \nabla^{\mu}\mathring{L}_{\mu}^{\nu}f_{\nu}).
	\end{split}
\end{equation}
We set $v_3 = v_2 + r^{n-s+3}\psi_3$.
It then follows from Proposition 3.6 in \cite{gz03} and Proposition \ref{expandprop} that
\begin{align*}
	\hat{P}_1^{SY} &= 0\\
	\hat{P}_2^{SY} &= -\Delta_k + \frac{n - 2}{4(n - 1)}(R - |\mathring{L}|_k^2)\\
	\intertext{If $n = 2$,}
	\hat{Q}_2^{SY} &= \frac{1}{2}(R - |\mathring{L}|_{k}^2).
	\intertext{For $n \geq 3$,}
	\hat{P}_3^{SY} &= \mathring{L}^{\mu\nu}\nabla^2_{\mu\nu}f + \nabla^{\mu}\mathring{L}_{\mu}^{\nu}f_{\nu} + \frac{n - 3}{4(n - 2)}\left( \nabla^{\mu}\nabla^{\nu}\mathring{L}_{\mu\nu} - 
	2\mathring{L}^{\mu\nu}R_{\mu\nu} + \mathring{L}^{\mu\nu}\hat{\overline{R}}_{\mu\nu}\right)f.\\
	\intertext{And if $n = 3$,}
	\hat{Q}^{SY}_3 &= \frac{1}{2}\nabla^{\mu}\nabla^{\nu}\mathring{L}_{\mu\nu} - \mathring{L}^{\mu\nu}R_{\mu\nu} + \frac{1}{2}\mathring{L}^{\mu\nu}\hat{\overline{R}}_{\mu\nu}.
\end{align*}

Finally, we can use Lemma \ref{uexplem} to translate these results into formulas for our general metric $\bar{g}$. We obtain the following.
\begin{theorem}
	\label{explicthm}
	For any dimension,
	\begin{align*}
		P_2^{SY}f &= -\Delta_kf + \frac{n - 2}{4(n - 1)}(R_k - |\mathring{L}|_{k}^2)f \quad (n \geq 2)\\
		P_3^{SY}f &= \mathring{L}^{\mu\nu}\nabla^2_{\mu\nu}f\\
		&\quad+ \nabla^{\mu}\mathring{L}_{\mu}^{\nu}f_{\nu} + \frac{n - 3}{4(n - 2)}\left( \nabla^{\mu}\nabla^{\nu}\mathring{L}_{\mu\nu} - 
		2\mathring{L}^{\mu\nu}R_{\mu\nu} + \mathring{L}^{\mu\nu}\hat{\overline{R}}_{\mu\nu} + \frac{n - 1}{n}H|\mathring{L}|_k^2\right)f.
	\end{align*}
	If $n = 2$, then
	\begin{equation*}
		Q_2^{SY} = \frac{1}{2}(R_k - |\mathring{L}|_k^2).
	\end{equation*}
	For $n = 3$, we have the following.
	\begin{align*}
		\syQ_3 &=\nabla^{\mu}\nabla^{\nu}\mathring{L}_{\mu\nu} - 2\mathring{L}^{\mu\nu}R_{\mu\nu} + \mathring{L}^{\mu\nu}\overline{R}_{\mu\nu} + \frac{2}{3}H|\mathring{L}|_k^2.
	\end{align*}
	Observe that $P_2^{SY}$ is the conformal Laplacian plus an extrinsic pointwise conformal invariant.
\end{theorem}

We next wish to compute the coefficients in Theorem \ref{globthm1} in the cases $n = 2$ and $n = 3$ to express the renormalized volume in terms of the scattering operator. On the basis of (\ref{psi3}) we have
\begin{equation*}
	\begin{split}
		v_3 =& \hr^{n - s}f + \hr^{n - s + 2}\left[ \frac{n - s}{4(n - 1)(n + 2 - 2s)}(R - |\mathring{L}|_k^2)f - \frac{1}{2(n + 2 - 2s)}\Delta_kf \right]\\
	&+\hr^{n - s + 3}\left[ \frac{n-s}{3(n-2)(n + 3-2s)}(-\nabla^{\mu}\nabla^{\nu}\mathring{L}_{\mu\nu}+2\mathring{L}^{\mu\nu}R_{\mu\nu}-\mathring{L}^{\mu\nu}\hat{\overline{R}}_{\mu\nu})\right.\\
	&- \left. \frac{2}{3(n+3-2s)}(\mathring{L}^{\mu\nu}\nabla_{\mu\nu}^2f + \nabla^{\mu}\mathring{L}_{\mu}^{\nu}f_{\nu})\right].
	\end{split}
\end{equation*}
Recall that $r$ (as opposed to $\hr$) is the distance function from $M$ with respect to $\bar{g}$.
We wish to express $v_2$ in terms of $r$ instead of $\hr$, and for this purpose, we want to expand $\hr^{\alpha}$ for a real number $\alpha$. Let $\alpha \in \mathbb{R}$. Recall that
we defined $\tu$ by $u = r\tu$. Then
\begin{equation*}
	\hr = e^{\omega}u = re^{\omega}\tu,
\end{equation*}
where $\omega$ is as in Lemma \ref{omegalem}. Thus,
\begin{equation*}
	\hr^{\alpha} = r^{\alpha}e^{\alpha(\omega + \log\tu)}.
\end{equation*}
Now, $\tu = 1 + r\tu_r + \frac{1}{2}r^2\tu_{rr} + \cdots$, where we set $\tu_r = \partial_r\tu|_{r = 0}$, etc.
A standard calculation shows that
\begin{equation*}
	\log \tu = r\tu_r + \frac{1}{2}r^2(\tu_{rr} - \tu_r^2) + \frac{1}{6}r^3(\tu_{rrr} - 3\tu_r\tu_{rr} + 2\tu_r^3) + \cdots.
\end{equation*}
Therefore,
\begin{equation*}
	\omega + \log \tu = r(\tu_r + \omega_r) + \frac{1}{2}r^2(\tu_{rr} - \tu_r^2 + \omega_{rr}) + \frac{1}{6}r^3(\tu_{rrr} - 3\tu_r\tu_{rr} + 2\tu_r^3 + \omega_{rrr}) + \cdots.
\end{equation*}
We obtain
\begin{align*}
	\exp[\alpha(\omega + \log \tu)] &= 1 + r(\alpha \tu_r + \alpha \omega_r) + \\
	&\quad + r^2\left[ \frac{\alpha}{2}\tu_{rr} - \frac{\alpha}{2}\tu_r^2 + \frac{\alpha}{2}\omega_{rr} + \frac{\alpha^2}{2}\tu_r^2 + \alpha^2\tu_r\omega_r + \frac{\alpha^2}{2}\omega_r^2 \right]\\
	&\quad + r^3\left[ \frac{\alpha}{6}\tu_{rrr} - \frac{\alpha}{2}\tu_r\tu_{rr}+\frac{\alpha}{3}\tu_r^3 + \frac{\alpha}{6}\omega_{rrr}+\frac{\alpha^2}{2}\tu_r\tu_{rr}\right.\\
	&\quad -\frac{\alpha^2}{2}\tu_r^3 + \frac{\alpha^2}{2}\tu_r\omega_{rr} + \frac{\alpha^2}{2}\tu_{rr}\omega_r - \frac{\alpha^2}{2}\tu_r^2\omega_r + \frac{\alpha^2}{2}\omega_r\omega_{rr}\\
	&\quad \left. + \frac{\alpha^3}{6}\tu_r^3 + \frac{\alpha^3}{2}\tu_r^2\omega_r + \frac{\alpha^3}{6}\omega_r^3 \right] + O(r^4).
\end{align*}
Supposing now that $\alpha = n - s + j$ ($j \geq 0$), and using Lemmas \ref{uexplem} and \ref{omegalem}, a tedious calculation yields
\begin{align*}
	\hr^{n - s + j} &= r^{n - s + j} + r^{n - s + j + 1}\left( \frac{n - s + j}{2n}H \right)\\
	&\quad + r^{n - s + j + 2}\left[ \frac{5n^2 - 8sn + 8jn + 3n + 3s^2 - 6js - 3s + 3j^2 + 3j}{24n^2}H^2\right.\\
	&\quad + \left.\frac{n - s + j}{12n}\overline{R} - \frac{n - s + j}{12(n - 1)}R + \frac{n - s + j}{12(n - 1)}
	|\mathring{L}|_k^2\right]\\
	&\quad + r^{n - s + j + 3}\left[ \frac{s -n - j}{24n}\Delta_kH + \frac{n - s + j}{24(n - 2)}\nabla^{\mu}\nabla^{\nu}\mathring{L}_{\mu\nu} + \frac{n - s + j}{24(n - 2)}\mathring{L}^{\mu\nu}\overline{R}_{\mu\nu}\right.\\
	&\quad + \frac{s - n - j}{12(n - 2)}\mathring{L}^{\mu\nu}R_{\mu\nu} + \frac{n - s + j}{48n}\partial_r\overline{R}\numberthis\label{bigexpeq}\\
	&\quad + \frac{(n - s + j)(4n^2-4sn+4jn+6n+s^2-2js-3s+j^2+3j+2)}{48n^3}H^{3}\\
	&\quad + \frac{(n - s + j)(3n-2s+2j+2)}{48n^2}H\overline{R} - \frac{(n - s + j)(3n-2s+2j+3)}{48n(n-1)}HR\\
	&\quad + \left.\frac{(n-s+j)(5n^2-2sn+2jn-7n+4s-4j-4)}{48n(n-1)(n-2)}H|\mathring{L}|_k^2\right] + O(r^{n-s+j+4}).
\end{align*}
So for $n = 2$, we get
\begin{equation*}
	\hr^{2 - s} = r^{2 - s}\left(1 + \frac{2 - s}{4}rH + r^2\left[ \frac{26 - 19s + 3s^2}{96}H^2 + \frac{2 - s}{24}\overline{R} - \frac{s - 2}{12}R + \frac{2 - s}{12}|\mathring{L}|_k^2 \right]\right).
\end{equation*}
Now taking $f = 1$ and $n = 2$, we find
\begin{align*}
	v_2 &= r^{2 - s} + \frac{2 - s}{4}r^{3 - s}H + r^{4 - s}\left[ \frac{3s^2-19s + 26}{96}H^2 + \frac{2-s}{24}\overline{R} + \frac{s - 2}{12}R + \frac{2 - s}{12}|\mathring{L}|_k^2 \right]\\
	&+ O(r^{5 - s}).
\end{align*}
So
\begin{align*}
	a_1(s) &= \frac{2 - s}{4}H\\
	a_2(s) &= \frac{3s^2-19s + 26}{96}H^2 + \frac{2-s}{24}\overline{R} + \frac{s - 2}{12}R + \frac{2 - s}{12}|\mathring{L}|_k^2.
\end{align*}
Thus,
\begin{align*}
	a_1'(2) &= -\frac{1}{4}H\\
	a_2'(2) &= -\frac{7}{96}H^2 -\frac{1}{24}\overline{R} + \frac{1}{12}R - \frac{1}{12}|\mathring{L}|_k^2.
\end{align*}
Now, by (\cite{gg19}), we have
\begin{align}
	\label{vol1eq}
	v^{(1)} &= \frac{1-n}{2n}H\\
	v^{(2)} &= \frac{n - 5}{12(n - 1)}(R - |\mathring{L}|_k^2) + \frac{n - 2}{24n^2}\left( (n - 3)H^2 - 2n\overline{R} \right).
\end{align}
So in this case, we have $v^{(1)} = -\frac{1}{4}H$. Here $v^{(1)}$ is the first renormalized volume coefficient, not our solution $v$ to the scattering equation.
Thus, using Theorem \ref{globthm1}, we find that for $n = 2$,
\begin{equation}
	\label{V2form}
	V(X,g,\bar{g}) = -\oint_M\left( \left.\frac{d}{ds}\right|_{s = 2}S(s)1 \right)dv_k - \frac{1}{96}\oint_M(8R - 4\overline{R} - 8|\mathring{L}|_k^2 - 3H^2)dv_k.
\end{equation}
Observe that if $g$ happens to be Einstein and $\bar{g}$ is a geodesic compactification (so that $\bar{g} = \hat{\bar{g}}$ -- which in this context is smooth), then
the last integral vanishes, since $\mathring{L}$ and $H$ vanish in this case, and $2R - \overline{R} = 0$ (by (\ref{n3geo2eq})). Thus, this result is consistent with
Theorem 4.1 of \cite{cqy08}.

We next turn to $n = 3$. Using (\ref{bigexpeq}) for $j = 0, 1$, and $2$ sequentially, and putting these into the formula for $v_3$, as well as using Lemma \ref{confchangelem}, we get (with $f = 1$)
\begin{equation}
	\label{V3form}
	\begin{split}
		v_3 =& r^{3 - s} + r^{4 - s}\left( \frac{3 - s}{6}H \right)\\
		&+r^{5 - s}\left[ \frac{54 - 27s + 3s^2}{216}H^2 + \frac{3 - s}{36}\overline{R} + \frac{(s - 3)(1 - s)}{12(5 - 2s)}R + \frac{(s - 3)(s - 1)}{12(5 - 2s)}|\mathring{L}|_k^2 \right]\\
		&+r^{6 - s}\left[ \frac{s - 3}{72}\Delta_kH - \frac{s + 1}{24}\nabla^{\mu}\nabla^{\nu}\mathring{L}_{\mu\nu} - \frac{s + 1}{24}\mathring{L}^{\mu\nu}\overline{R}_{\mu\nu} + \frac{s + 1}{12}\mathring{L}^{\mu\nu}
		R_{\mu\nu}\right.\\
		&+ \frac{3 - s}{144}\partial_r\overline{R} + \frac{(3 - s)(56 - 15s + s^2)}{1296}H^3 + \frac{(3 - s)(11-2s)}{432}H\overline{R}\\
		&+ \left.\frac{(s - 3)(2s^2-14s+15)}{144(5 - 2s)}HR - \frac{2s^3-28s^2+69s - 25}{144(5-2s)}H|\mathring{L}|_k^2\right] + O(r^{7 - s}).
	\end{split}
\end{equation}

It quickly follows that
\begin{align*}
	a_1'(3) &= -\frac{1}{6}H\\
	a_2'(3) &= -\frac{1}{24}H^2 - \frac{1}{36}\overline{R} + \frac{1}{6}R - \frac{1}{6}|\mathring{L}|_k^2\\
	a_3'(3) &= \frac{1}{72}\Delta_kH - \frac{1}{24}\nabla^{\mu}\nabla^{\nu}\mathring{L}_{\mu\nu} - \frac{1}{24}\mathring{L}^{\mu\nu}\overline{R}_{\mu\nu}+\frac{1}{12}\mathring{L}^{\mu\nu}R_{\mu\nu} - \frac{1}{144}\partial_r
	\overline{R}\\
	&\quad-\frac{5}{324}H^3 - \frac{5}{432}H\overline{R} + \frac{1}{16}HR - \frac{13}{144}H|\mathring{L}|_k^2.
\end{align*}
Using the formulae (\ref{vol1eq}) for the renormalized volume coefficients, we finally obtain
\begin{align*}
	V(X,g,\bar{g}) &= \oint_M\left[ -\left.\frac{d}{ds}\right|_{s = 3}S(s)1 - \frac{13}{432}HR + \frac{5}{1296}H\overline{R} + \frac{1}{162}H^3 + \frac{25}{432}H|\mathring{L}|_k^2\right.\\
		&\quad\left.-\frac{1}{72}\Delta_kH + \frac{1}{24}\nabla^{\mu}\nabla^{\nu}\mathring{L}_{\mu\nu} + \frac{1}{24}\mathring{L}^{\mu\nu}\overline{R}_{\mu\nu} + \frac{1}{12}\mathring{L}^{\mu\nu}R_{\mu\nu}
		+ \frac{1}{144}\partial_r\overline{R}\right]dv_k.
\end{align*}

Theorem \ref{gbprop} and Corollary \ref{gbcor} now follow directly from this computation and the main result of \cite{gg19}.

\appendix
\section{Analysis}
\label{proofapp}
We present here the proof of Lemma \ref{geodlem} and sketch the proof of
Theorem \ref{poissonthm}. Both entail adapting standard results for smooth AH metrics to the polyhomogeneous setting with more careful discussion of regularity.

\renewcommand{\thesubsection}{\thesection.\Roman{subsection}}

\subsection{Normal Form}
For $(X,\bar{g})$ a smooth manifold with boundary, let $r$ be the distance function to the boundary. 
We let $\Diff_b(X)$ be the ring of differential operators generated by vector fields $V$ that are tangent to the boundary. In local coordinates
$(r,x^{\mu})$ ($1 \leq \mu \leq n$) near the boundary $M$, $\Diff_b(X)$ is generated over $C^{\infty}$ by $r\frac{\partial}{\partial r}$ and $\frac{\partial}{\partial x^{\mu}}$.

For $p \geq 2$, we let $\mathcal{D}_p$ be the conormal functions 
\begin{equation*}
	\mathcal{D}_p(X) = \left\{ u \in C^p(X):Lu \in C^p(X) \text{ for all } L \in \Diff_b(X) \right\}.
\end{equation*}
The examples of interest
to us are smooth functions and those with asymptotic expansions in $r$ and $r^q\log(r)$, $q \geq p + 1$, with smooth coefficients. The following lemma is easy.

\begin{lemma}
	Let $\omega \in \mathcal{D}_p$. Then $r\omega \in \mathcal{D}_{p + 1}$.
\end{lemma}

We prove the following variation on the existence and uniqueness of solutions to first-order scalar noncharacteristic PDEs. We let $x^0 = r$, and locally extend any coordinate chart
$(x^1,\cdots,x^n)$ on a neighborhood $U \subseteq M$ to be coordinates along with $r$ on a neighborhood $\widetilde{U}$ of $U$ in $X$, by the geodesic identification $\widetilde{U} \approx [0,\varepsilon)_r \times U$,
where $\approx$ denotes diffemorphism.
We then let $(\xi_0,\cdots,\xi_n)$ be the corresponding natural coordinates on $T^*X$ on $\widetilde{U}$.
\begin{proposition}
	\label{noncharprop}
	Let $F \in C^{p}(T^*X \times \mathbb{R})$ be such that, for any smooth one-form $\eta \in \Omega^1(X), u \in C^{\infty}(X)$, the function $x \mapsto F(x,\eta(x),u(x)) \in \mathcal{D}_p(X)$, where $p \geq 2$. Suppose
	that, in any coordinate system as above, $\left.\frac{\partial F}{\partial \xi_0}\right|_M > 0$. Finally, suppose given $\varphi, \tau \in C^{\infty}(M)$
	such that, for every $q \in M$, $F(q,\tau dr + d\varphi,\varphi) = 0$.
	Then there exists a neighborhood $\mathcal{V}$ of $M$ and a unique solution
	$\omega$ to $F(x,d\omega,\omega) = 0$ on $\mathcal{V}$ such that $\omega|_M = \varphi$. Moreover, $\omega \in \mathcal{D}_p$.
\end{proposition}
\begin{proof}
	Everything except the last statement is standard. We present an adaptation of the usual proof which preserves conormality. We work locally in a coordinate chart.
	Recall that the ordinary proof relies on converting the PDE
	to a first-order system of ODEs representing a characteristic flow off of $M$, and parametrized by time $t$. We first eliminate $t$ and replace it by the first coordinate $r$, which is
	permissible because $\frac{\partial r}{\partial t}|_{t = 0} > 0$ by the noncharacteristic hypothesis. We thus begin by considering the following system of ODEs.
	\begin{align*}
		\frac{d\xi_j}{dr} &= \frac{\frac{\partial F}{\partial x^j} + \xi_j\frac{\partial F}{\partial y}}{\frac{\partial F}{\partial \xi_0}} \quad (0 \leq j \leq n)\\
		\frac{dx^{\mu}}{dr} &= \frac{\frac{\partial F}{\partial \xi_{\mu}}}{\frac{\partial F}{\partial \xi_0}} \quad (1 \leq \mu \leq n)\\
		\frac{dy}{dr} &= -\xi_j\frac{\frac{\partial F}{\partial \xi_j}}{\frac{\partial F}{\partial \xi_0}},
	\end{align*}
	with initial conditions $x^{\mu} = \zeta^{\mu}$ (some $\zeta^{\mu}$), $\xi_{\mu}(0) = \partial_{\mu}\varphi(\zeta^1,\cdots,\zeta^n)$, $y(0) = \varphi(\zeta^1,\cdots,\zeta^n)$, and with
	$\xi_0(0) = \tau$.
	Letting $z = (x^1,\cdots,x^n,\xi_0,\cdots,\xi_n,y)$, this system may be written
	\begin{equation*}
		\frac{dz}{dr} = G(r,z),
	\end{equation*}
	where by hypothesis $G$ is $C^{p - 1}$. Moreover, for any multi-index $\alpha$, $\frac{\partial G}{\partial x^{\alpha}}$ is $C^{p - 1}$ (where $\alpha$ contains no 0's, i.e., no $r$-derivatives); and so
	likewise is $(r\partial_r)^kG$ for any $k$.

	The system is equivalent to the integral equation
	\begin{equation}
		\label{firstinteq}
		z = z_0 + \int_0^r G(s,z(s))ds,
	\end{equation}
	where $z_0$ contains the initial values written above. This equation, of course, has a $C^1$ solution by standard theory. Now we view $z$ as a function both of $r$ and of $(\zeta^1,\cdots,\zeta^n)$. Formally
	differentiating (\ref{firstinteq}), we get
	\begin{equation*}
		\frac{d}{dr}\frac{\partial z}{\partial \zeta^{\mu}} = \frac{\partial z_0}{\partial \zeta^{\mu}} + \int_0^r\frac{\partial G}{\partial z^{a}}\frac{\partial z^{a}}{\partial \zeta^{\mu}}
	\end{equation*}
       (where we use the index $a$ as an index for the components of $z$). Now, the integrand here is still $C^{p - 1}$ by hypothesis; and so by a standard argument, and its accompanying induction
       (see, e.g., section 13 of \cite{wal98}), $\frac{\partial^{|\alpha|}z}{\partial \zeta^{\alpha} }$ is $C^p$.
       On the other hand, using the identity
       \begin{equation*}
	       \frac{d}{dr}(r\frac{dz}{dr}) = r\frac{d^2z}{dr^2} + \frac{dz}{dr}
       \end{equation*}
       and (\ref{firstinteq}), along with the easily-verified fact that $H = \int_0^rG(s,z)ds$ is $C^p$ along with $\left( r\frac{d}{dr} \right)^kH$ for any $k$, we conclude by induction that
       $\left( r \frac{\partial}{\partial r} \right)^k\left(\frac{\partial^{|\alpha|}}{\partial \zeta^{\alpha}}\right)z$ is $C^p$.

       We now have functions $x^1(r,\zeta^1,\cdots,\zeta^n),\cdots,x^n(r,\zeta^1,\cdots,\zeta^n)$. We wish to invert the dependence of $x$ on $\zeta$. For notational simplicity, we assume by restriction
       and rescaling if necessary that $\zeta^1,\cdots,\zeta^n$ take values in all of $\mathbb{R}^n$. Let $\Phi(r,\zeta^1,\cdots,\zeta^n) = (x^1,\cdots,x^n)$, and define
       $Z_{r_0}:\mathbb{R}^n \to \mathbb{R}^n$ by $Z_{r_0}(\zeta) = \Phi(r_0,\zeta)$. Now $\Phi$ is $C^p$, and $DZ_0$ is the identity, so by restricting $r_0$ if necessary, we may assume
       that $Z_{r}$ is invertible. In fact, for each $r_0$ it is a $C^{\infty}$ diffeomorphism onto its image, and thus has $C^{\infty}$ inverse $Z_{r_0}^{-1}$. That this inverse is $C^p$ in $r$ follows
       from the implicit function theorem.

       To show that $(r\partial_r)^kZ_{r}^{-1}$ is likewise $C^p$, we write
       \begin{equation*}
	       \Phi(r,Z_r^{-1}(x^1,\cdots,x^n)) = (x^1,\cdots,x^n),
       \end{equation*}
       and then differentiate both sides with respect to $r$. We obtain
       \begin{equation*}
	       \frac{\partial \Phi}{\partial r} + D_{\zeta}Z_r\left( \frac{\partial Z^{-1}}{\partial x^{\mu}}\frac{\partial x^{\mu}}{\partial r} + \frac{\partial Z_r^{-1}}{\partial r} \right)
	       = \left( \frac{\partial x^1}{\partial r},\cdots,\frac{\partial x^n}{\partial r} \right).
       \end{equation*}
       Since $D_{\zeta}Z_r$ is invertible, we can solve this for $\frac{\partial Z_r^{-1}}{\partial r}$, and in particular exhibit
       $r\frac{\partial Z_r^{-1}}{\partial r}$ as a combination of terms already known to be $C^p$.

       Thus, if we set $\omega(r,x) = y(r,Z_r^{-1}(x))$, then $\omega$ is a solution of our PDE -- just as in the usual case, it is a solution at $r = 0$, and we show that
       $\frac{d}{dr}F(x,d\omega,\omega) = 0$. Moreover, $\omega \in \mathcal{D}_p$.
\end{proof}

We now can prove Lemma \ref{geodlem}.

\begin{proof}[Proof of Lemma \ref{geodlem}]
	Let $r$ be the distance to $M$ with respect to $\bar{g}$. We wish to find $\hat{r}$ so that $|d\hr|_{\hr^2g} \equiv 1$ on some neighborhood of the boundary. As in the usual proof (e.g. \cite{g99}), 
	write $\hr = ue^{\omega} = r\tu e^{\omega},$ where $u$ is the singular Yamabe function and $\tu = \frac{u}{r}$. Recall that $\tu = 1 + O(r)$. We thus wish to find $\omega$ so that $\omega|_M = 0$ and so that
	\begin{equation*}
		2\grad_{\bar{g}}(u)(\omega) + u|d\omega|_{\bar{g}}^2 = \frac{1 - |du|_{\bar{g}}^2}{u}.
	\end{equation*}
	In the smooth case, one directly solves this by observing that it is a first-order noncharacteristic equation. However, doing so in this form would not give optimal regularity, which we want.
	Using the fact that $u = r\tu$, we can rewrite this as follows:
	\begin{equation}
		\label{normformeqexp}
		2\partial_r\omega + 2r\tu^{-1}\bar{g}^{ij}\tu_i\omega_j + r\bar{g}^{ij}\omega_{i}\omega_{j} = \frac{1 - \tu^2 - r^2|d\tu|_{\bar{g}}^2 - 2r\tu\partial_r\tu}{r\tu^2}.
	\end{equation}
	The right-hand side is $C^{n - 1}$, and in particular, taking $r = 0$, we may solve for $\partial_r\omega|_{r = 0}$. We can then differentiate (\ref{normformeqexp}) iteratively, and at each stage,
	$\partial_r\omega$ is expressed in terms of already-determined quantities. Observe, however, that by (\ref{usyexp}), the right-hand side contains a term of the form $2\mathcal{L}r^n\log(r)$.
	We can handle this by adding a term of the form $Ar^{n + 1}\log(r)$ to our formal expansion of $\omega$ ($A \in C^{\infty}(M)$), and iterating this procedure, we can find some $\omega^0$ of the form
	\begin{equation*}
		\omega^0 = a_1r + a_2r^2 + \cdots + a_nr^n + Ar^{n + 1}\log(r) + a_{n + 1}r^{n + 1} + \cdots,
	\end{equation*}
	where each $a_j \in C^{\infty}(M)$, and such that equation (\ref{normformeqexp}) is satisfied through order $O(r^{n + 1})$ by $\omega^0$. 
	Now set $\omega = \omega^0 + \Omega$, and substitute this into (\ref{normformeqexp}). We obtain the equation
	\begin{equation*}
		2\partial_r\Omega + 2r\tu^{-1}\bar{g}^{ij}\tu_i\Omega_j + 2r\bar{g}^{ij}\omega^0_i\Omega_j + r\bar{g}^{ij}\Omega_i\Omega_j = O(r^{n + 1}),
	\end{equation*}
	where the right-hand side, in particular, is a function in $\mathcal{D}_{n + 1}$. Then since $ru_i, r\omega^0_i \in C^n$, this is a noncharacteristic first-order equation with $C^n$ coefficients,
	and in fact the differential operator is $n$-conormal. Thus, by Proposition \ref{noncharprop}, there is a unique solution $\Omega \in \mathcal{D}_n$, and so we get a unique solution
	$\omega = \omega^0 + \Omega$ to (\ref{normformeqexp}), and $\omega \in \mathcal{D}_n$ as well. Consequently, $\hat{r} = r\tilde{u}e^{\omega} \in \mathcal{D}_{n + 1}$. We construct the diffeomorphism $\psi$,
	as always, by following the flow lines of $\grad_{\bar{g}}\hat{r}$. This is a $C^n$ vector field, so the result follows.
\end{proof}
Observe that we actually show rather more than claimed -- specifically, that the diffeomorphism is conormal -- but the lemma is all we need for our purposes.

\subsection{Meromorphic Extension of the Resolvent}
Theorem \ref{poissonthm} is proved in \cite{gz03} in the context of smooth AH metrics. The proof proceeds by first constructing an infinite-order formal solution $u_f$ satisfying
\[(\Delta_g + s(n - s))u_f = O(r^{\infty}),\]
and then using the following theorem from \cite{mm87} (see also \cite{g05}).
\begin{theorem}
	Let $(X,g)$ be an asymptotically hyperbolic manifold. Then the resolvent $R(s) = (\Delta_g + s(n - s))^{-1}:L^2_g(\mathring{X}) \to L^2_g(\mathring{X})$ for $s > n$ has an extension
	$R(s):r^{\infty}C^{\infty}(X) \to r^sC^{\infty}(X)$ that is holomorphic on $\mathbb{C} \setminus \Gamma$ with $\Gamma \subset \mathbb{C}$ discrete. In particular,
	$R(s)$ is meromorphic on the right half-plane.
\end{theorem}
Thus, we need the same theorem in the case where $g$ is not smooth AH, but rather polyhomogeneous, and where the codomain of $R(s)$ may likewise be only polyhomogeneous. Fortunately, the needed modifications to the proof
are slight. The proof of \cite{mm87} proceeds by first proving the result on hyperbolic space, and in fact obtaining an explicit formula for the Schwartz kernel of the resolvent there. On a general asymptotically hyperbolic
space, the proof proceeds in three steps. First, ordinary elliptic analysis is used to produce an interior solution with poor boundary regularity; then, the ``normal operator'' of the Laplacian at a fixed point of the boundary is
analyzed, and shown to coincide with the Laplacian on hyperbolic space, so that the result there can be used to obtain better regularity. Finally, the indicial operator at each point is used to obtain optimal regularity
via formal expansion and Borel's lemma. The first two steps go through exactly the same if the metric $g$ is polyhomogeneous and (say) $C^n$. And the last step, likewise, will be the same except that
logarithmic terms from the metric may appear on the right hand side of the order-by-order construction, and need to be corrected by including logarithmic terms in the solution. The order at which they appear can be
computed formally using the indicial operator.

In particular, the construction of the scattering operator proceeds exactly as in \cite{gz03}, except that because smooth compactifications of the metric have a log term appearing of the form
$r^{n + 1}\log(r)$, expansions of $\mathcal{P}(s)f$ starting at $r^{n - s}$ will have a term of the form $r^{n - s + (n + 1)}\log(r)$; to avoid the appearance of a log term at or before the second indicial root $s$, therefore
(which would complicate the analysis), we require $s < 2n - s + 1$, or $s < n + \frac{1}{2}$.

\bibliographystyle{alpha}
\bibliography{sys}

\newcommand{\noop}[1]{}
\begin{thebibliography}{CdMG11}

\bibitem[ACF92]{acf92}
L.~Andersson, P.~T. Chru\'{s}ciel, and H.~Friederich.
\newblock On the regularity of solutions to the {Y}amabe equation and the
  existence of smooth hyperboloidal initial data for {E}instein's field
  equations.
\newblock {\em Comm. Math. Phys.}, 149(3):587--612, 1992.

\bibitem[AM88]{am88}
P.~Aviles and R.~C. McOwen.
\newblock Complete conformal metrics with negative scalar curvature in compact
  {R}iemannian manifolds.
\newblock {\em Duke Math. J.}, 56(2):395--398, 1988.

\bibitem[CdMG11]{cg11}
S.-Y.~A. Chang and M.~d.~M.~Gonzalez.
\newblock Fractional {L}aplacian in conformal geometry.
\newblock {\em Advances in Mathematics}, 226(2):1410--1432, 2011.

\bibitem[CQY08]{cqy08}
S.-Y.~A. Chang, J.~Qing, and P.~Yang.
\newblock Renormalized volumes for conformally compact {E}instein manifolds.
\newblock {\em J. Math. Sci. (N.Y.)}, 149(6):1755--1769, 2008.

\bibitem[FG02]{fg02}
C.~Fefferman and C.~R. Graham.
\newblock $q$-curvature and {P}oincar\'{e} metrics.
\newblock {\em Math. Res. Lett.}, 9:139--151, 2002.

\bibitem[GG19]{gg19}
C.~R. Graham and M.~J. Gursky.
\newblock Chern-{G}auss-{B}onnet formula for singular {Y}amabe metrics in
  dimension four.
\newblock {\em preprint. ar{X}iv:1902.01562}, pages 1--26, 2019.

\bibitem[GL91]{gl91}
C.~R. Graham and J.~M. Lee.
\newblock Einstein metrics with prescribed conformal infinity on the ball.
\newblock {\em Adv. Math.}, 87:186 -- 225, 1991.

\bibitem[Gra00]{g99}
C.~R. Graham.
\newblock Volume and area renormalizations for conformally compact {E}instein
  metrics.
\newblock {\em Suppl. Rendiconti Circolo Mat. Palermo}, 63:31--42, 2000.

\bibitem[Gra17]{g17}
C.~R. Graham.
\newblock Volume renormalization for singular {Y}amabe metrics.
\newblock {\em Proceedings of the AMS}, 145(4):1781--1792, 2017.

\bibitem[Gui05]{g05}
C.~Guillarmou.
\newblock Meromorophic properties of the resolvent on asymptotically hyperbolic
  manifolds.
\newblock {\em Duke Math. J.}, 129(1):1--37, 2005.

\bibitem[GW14]{gw14}
A.~R. Gover and A.~Waldron.
\newblock Boundary calculus for conformally compact manifolds.
\newblock {\em Indiana Univ. Math. J.}, 63(1):119--163, 2014.

\bibitem[GW15]{gw15}
A.~R. Gover and A.~Waldron.
\newblock Conformal hypersurface geometry via a boundary
  {L}oewner-{N}irenberg-{Y}amabe problem.
\newblock {\em Preprint:ar{X}iv:1506.02723}, 2015.

\bibitem[GW17]{gw17}
A.~R. Gover and A.~Waldron.
\newblock Renormalized volume.
\newblock {\em Comm. Math. Phys.}, 354(3):1205--1244, 2017.

\bibitem[GZ03]{gz03}
C.~R. Graham and M.~Zworski.
\newblock Scattering matrix in conformal geometry.
\newblock {\em Invent. Math.}, 152:89 -- 118, 2003.

\bibitem[HS98]{hs98}
M.~Hennington and K.~Skenderis.
\newblock The holographic {W}eyl anomaly.
\newblock {\em Journal of High Energy Physics}, 7:23, 1998.

\bibitem[JO21]{jo21}
A.~Juhl and B.~Orsted.
\newblock Residue families, singular {Y}amabe problems and extrinsic conformal
  {L}aplacians.
\newblock {\em preprint: ar{X}iv: 2101.09027}, pages 1--115, 2021.

\bibitem[LN74]{ln74}
C.~Loewner and L.~Nirenberg.
\newblock Partial differential equations invariant under conformal or
  projective transformations.
\newblock In {\em Contributions to analysis (a collection of papers dedicated
  to {L}ipman {B}ers)}, pages 245--272. Academic Press, New York, 1974.

\bibitem[McK18]{m18a}
S.~McKeown.
\newblock Formal theory of cornered asymptotically hyperbolic {E}instein
  metrics.
\newblock {\em J. Geom. Anal. (to appear)}, pages 1--53, 2018.
\newblock arXiv:1708.02390.

\bibitem[MM87]{mm87}
R.~R. Mazzeo and R.~B. Melrose.
\newblock Meromorphic extension of the resolvent on complete spaces with
  asymptotically constant negative curvature.
\newblock {\em J. Func. Anal.}, 75:260 -- 310, 1987.

\bibitem[Wal98]{wal98}
W.~Walter.
\newblock {\em Ordinary Differential Equations}.
\newblock Springer, New York, 1998.

\end{thebibliography}

\end{document}